\documentclass[12pt]{amsart}

\usepackage{amssymb,amsfonts}
\usepackage{graphicx}
\usepackage{pgf,tikz}
\pgfdeclarelayer{nodelayer}
\pgfdeclarelayer{edgelayer}
\pgfsetlayers{edgelayer,nodelayer,main}
\tikzset{none/.style={thick}}
\usetikzlibrary{arrows.meta}

\usepackage{pb-diagram,pb-xy} 
\usepackage[all,cmtip,arrow]{xy} 
\usepackage[margin=1in]{geometry}  
\usepackage{hyperref} 

\numberwithin{equation}{section}

\theoremstyle{plain}
\newtheorem{theorem}[equation]{Theorem}

\newtheorem{lemma}[equation]{Lemma}

\newtheorem{corollary}[equation]{Corollary}
\newtheorem{proposition}[equation]{Proposition}

\newtheorem{prop}[equation]{Proposition}

\theoremstyle{definition}
\newtheorem{definition}[equation]{Definition}
\newtheorem{notation}[equation]{Notation}

\newtheorem{remark}[equation]{Remark}

\newtheorem{example}[equation]{Example}

\newcommand{\isom}{\cong}                       
\newcommand{\homeq}{\simeq}                     
\newcommand{\smsh}{\wedge}                      

\newcommand{\Smsh}{\bigwedge}                   

\newcommand{\R}{\mathbb{R}}                     

\DeclareMathOperator*{\hocolim}{hocolim}

\DeclareMathOperator*{\holim}{holim}
\newcommand{\Map}{\operatorname{Map} }

\newcommand{\un}[1]{\underline{#1}}

\newcommand{\based}{\mathsf{Top}_*}
\newcommand{\spaces}{\mathsf{Top}}
\newcommand{\spectra}{{\mathsf{Sp}}}              

\newcommand{\weq}{\; \tilde{\longrightarrow} \;}      
\newcommand{\lweq}{\; \tilde{\longleftarrow} \;}      
\newcommand{\into}{\hookrightarrow}

\newcommand{\dual}{\mathbb{D}}                  

\begin{document}

\title{Koszul duality for topological $E_n$-operads}

\author{Michael Ching}
\address{Department of Mathematics and Statistics,
Amherst College,
Amherst, MA 01002, USA}
\email{mching@amherst.edu}

\author{Paolo Salvatore}
\address{Dipartimento di Matematica,
Universit\`{a} di Roma Tor Vergata,
Via della Ricerca Scientifica,
00133 Roma, Italy}
\email{salvator@mat.uniroma2.it}

\thanks{The first author was partially supported by the National Science Foundation through grant DMS-1709032. The second author acknowledges the 
MIUR Excellence Department Project awarded to the Department of Mathematics, University of Rome Tor Vergata, CUP E83C18000100006. The authors would also like to thank the Isaac Newton Institute for Mathematical Sciences, Cambridge, for support and hospitality during the programme \emph{Homotopy harnessing higher structures (HHH)} where much of this paper was written. That programme was supported by EPSRC grant number EP/R014604/1.}

\begin{abstract}
We show that the Koszul dual of an $E_n$-operad in spectra is $O(n)$-equivariantly equivalent to its $n$-fold desuspension. To this purpose we introduce a new $O(n)$-operad of Euclidean spaces $R_n$, the barycentric operad, that is fibred over simplexes and has homeomorphisms as structure maps; we also introduce its sub-operad of restricted little $n$-discs $D_n$, that is an $E_n$-operad. The duality is realized by an unstable explicit S-duality pairing $(F_n)_+ \wedge BD_n \to \bar{S}_n$, where $B$ is the bar-cooperad construction, $F_n$ is the Fulton-MacPherson $E_n$-operad, and the dualizing object $\bar{S}_n$ is an operad of spheres that are one-point compactifications of star-shaped neighbourhoods in $R_n$. We also identify the Koszul dual of the operad inclusion map $E_n \to E_{n+m}$ as the $(n+m)$-fold desuspension of an unstable operad map $E_{n+m} \to \Sigma^m E_n$ defined by May.
\end{abstract}

\maketitle

\section{Introduction}

As stated in the abstract the main result of this paper is the following theorem.

\begin{theorem} \label{thm:1}
Let $\mathbf{E}_n$ denote the stable (reduced) little $n$-discs operad, i.e.\ the operad of spectra formed by taking suspension spectra of the ordinary little discs operad of topological spaces. Then there is an $O(n)$-equivariant equivalence of operads of spectra
\[ K\mathbf{E}_n \homeq \Sigma^{-n}\mathbf{E}_n \]
between the Koszul dual of $\mathbf{E}_n$ and its $n$-fold operadic desuspension.
\end{theorem}

By an \emph{equivalence of operads} we mean a zigzag of operad morphisms (or often a single morphism), each of which is arity-wise a weak equivalence in some underlying category. In Theorem~\ref{thm:1} that underlying category consists of spectra formed from topological spaces with $O(n)$-action, and the weak equivalences are those $O(n)$-equivariant maps which are stable weak equivalences when forgetting the $O(n)$-actions, i.e.\ we are working in \emph{naive} $O(n)$-equivariant stable homotopy theory. At several other points in this paper, we employ equivalences of operads in the category of pointed spaces, where the underlying weak equivalences are the weak homotopy equivalences.

The purpose of this introduction is to review the main objects involved in Theorem~\ref{thm:1}, describe its significance, and summarize our approach to the theorem. We start by recalling the little $n$-discs operad.

\subsection*{Topological \texorpdfstring{$E_n$}{En}-operads}

The little $n$-discs/cubes topological operads were introduced by Boardman and Vogt in the 1970s in order to parametrize the natural operations on $n$-fold loop spaces coming from configuration spaces in $\R^n$, together with their compositions. Operations with $k$ inputs are parametrized by families of $k$ discs in the unit disc with disjoint interiors, and composition operations correspond to rescaling and gluing families of discs. May coined the word `operad' to describe this algebraic structure of composable operations with many inputs and one output. Since then operads in general symmetric
monoidal categories have been extensively studied. The objects acted upon by an operad $P$ are called $P$-algebras, and any (suitably cofibrant) operad equivalent to the little $n$-discs is called an `$E_n$-operad'. More generally $E_n$-operads and $E_n$-algebras make sense in any symmetric monoidal $\infty$-category. For example, an $E_1$-algebra is an object with a binary operation that is associative up to higher coherent homotopies. For $n > 1$, the operation in an $E_n$-algebra also possesses some degree of commutativity. The prototypical examples of topological $E_n$-algebras are $n$-fold iterated loop spaces.

\subsection*{Koszul duality for operads}

The idea of Koszul duality also originated in the 1970s, and is due to Priddy~\cite{priddy:1970}: to an algebra $A$ generated by quadratic operations and relations, once can contravariantly assign a `dual' $DA$. For certain algebras, designated `Koszul', there is an induced contravariant equivalence between the derived categories of $A$-modules and $DA$-modules satisfying suitable conditions. A version of Koszul duality for operads was first introduced by Ginzburg and Kapranov~\cite{ginzburg/kapranov:1994} in the context of operads of chain complexes. They constructed a contravariant functor
\[ \mathbf{D}: \mathsf{Op}(\mathsf{Ch}_\mathsf{k})^{op} \to \mathsf{Op}(\mathsf{Ch}_\mathsf{k}) \]
from the category of operads of chain complexes of vector spaces (over a field $\mathsf{k}$ of characteristic zero) to itself such that, subject to finiteness conditions, $\mathbf{D}(\mathbf{D}(P)) \homeq P$ for each such operad $P$.
They also constructed a contravariant functor from $P$-algebras to $\mathbf{D}(P)$-algebras. A fundamental example is $P = E_1$ in which case $\mathbf{D}(E_1)$ is equivalent to $E_1$ up to shift in dimension, and the associated functor on associative algebras coincides with the (dual of) the classical bar-cobar duality between associative algebras and coassociative coalgebras, for example as described by Moore~\cite{moore:1971}.

Another fundamental example arises for $P=Com$, the commutative operad, for which we have $\mathbf{D}(Com) \homeq Lie$, the Lie operad. The induced functor between algebras over these operads appears in Quillen's key work on rational homotopy theory~\cite{quillen:1969} giving two different algebraic models for simply-connected rational homotopy types: one based on differential graded Lie-algebras, and one on commutative differential graded coalgebras.

Ginzburg and Kapranov also introduced the notion of a \emph{Koszul} operad $P$: one for which the operad $\mathbf{D}(P)$ admits a particularly nice model, called the \emph{Koszul dual} of $P$. For example, the Lie and commutative operads are Koszul duals of each other.

Getzler and Jones~\cite{getzler/jones:1994} reworked some of the Ginzburg-Kapranov constructions in terms of an equivalence between operads and (connected) cooperads, still for chain complexes over a field $\mathsf{k}$ of characteristic zero. They also extended the self-duality for the associative operad to the homology of the $E_n$-operads for $n > 1$. That is, they constructed an equivalence
\begin{equation} \label{eq:en} \mathbf{D}(e_n) \homeq s^{-n}e_n \end{equation}
where $e_n = H_*(E_n,\mathsf{k})$ is the operad of graded vector spaces formed by the homology (with coefficients in $\mathsf{k}$) of the little $n$-disc operad, and $s^{-n}$ denotes the $n$-fold operadic desuspension given by suitable shifts in dimension (and sign changes). As with the previous examples, the Getzler and Jones duality induces bar and cobar constructions between $e_n$-algebras and $e_n$-coalgebras, known also as Poisson $n$-(co)algebras when $n > 1$, and Gerstenhaber algebras for
$n=2$.

It is the equivalence (\ref{eq:en}) that we generalize in this paper to an equivalence of underlying topological operads. Benoit Fresse made significant steps in this direction by proving in \cite{fresse:2011} a version of (\ref{eq:en}) with the (characteristic zero) homology operad $e_n$ replaced by an integral chain model $C_*(E_n)$ for the little disc operad, that is an $E_n$-operad in chain complexes.

\begin{theorem}[Fresse~\cite{fresse:2011}] \label{thm:fresse}
There is an equivalence of operads of chain complexes (of abelian groups)
\[ \mathbf{D}(C_*(E_n)) \homeq s^{-n}C_*(E_n) \]
where $\mathbf{D}$ is the extension of Ginzburg-Kapranov's dg-dual to integral chain complexes, as described in \cite{fresse:2004}.
\end{theorem}

This equivalence determines a contravariant endofunctor on the category of $E_n$-algebras in chain complexes that relates the $E_n$-algebras $C_*(\Omega^n X)$ and $C^*(X)$, for an $n$-connected space $X$ of finite type, see \cite{fresse:2010}, extending the Adams-Hilton \cite{adams/hilton:1956} classical duality for $n=1$.

All of the above discussion has been about algebraic operads, but this paper concerns operads and duality in a topological setting. The authors, independently in their Ph.D.~theses \cite{salvatore:1999} and \cite{ching:2005}, described a cooperad structure on the bar construction $BP$ of a reduced topological operad $P$. The Spanier-Whitehead dual of $BP$ is an operad $KP$ of spectra which has come to be known as the (derived) Koszul dual of $P$, despite more accurately being the analogue of the dg-dual $\mathbf{D}(P)$.

The first author showed in \cite{ching:2012} that, subject to finiteness conditions, the Koszul duality construction $K$ is self-adjoint, i.e.\ for a reduced operad $P$ of spectra, there is an equivalence of operads
\[ K(K(P)) \homeq P. \]
The close analogy between the derived Koszul dual $K$ and the dg-dual $\mathbf{D}$ led to the conjecture, made initially by the authors in 2005 and stated explicitly in \cite{ching:2012}, that the little disc operad $E_n$ (or rather its associated operad of spectra) satisfies an equivalence similar to that of Theorem~\ref{thm:fresse}. In this paper we prove that indeed this is the case: that is the content of Theorem~\ref{thm:1}.

We stress that the duality is unstable in nature, since it originates from an S-duality pairing of the form
\[ (E_n)_+ \wedge BE_n \to \bar{S}_n \]
where $\bar{S}_n$ is a certain operad of spheres, see Section~\ref{sec:susp} for details. This fits well with the work by Ayala and Francis on factorization homology \cite[\S3.2]{ayala/francis:2021}, where the authors construct the expected induced functor between unstable $E_n$-algebras and $E_n$-coalgebras, and remark that no explicit duality has been constructed on the operad level yet.

Since the conjecture on self duality of the $E_n$-operad was made, work of Lurie has put bar-cobar dualities in homotopy theory into a wider framework. In~\cite[\S5.2]{lurie:2017}, Lurie describes bar and cobar constructions between the monoids and comonoids in any monoidal $\infty$-category. Applying this work to the composition product of symmetric sequences yields another version of bar-cobar duality for operads (which can be viewed as monoids for the composition product).

Moreover, iterating the construction for monoids (or $E_1$-algebras), Lurie obtains, in \cite[5.2.5]{lurie:2017}, bar and cobar constructions between topological $E_n$-algebras and coalgebras (in an arbitrary $E_n$-monoidal $\infty$-category). We expect the result of this paper to be closely related to those constructions.

\subsection*{Suspension and desuspension for topological operads}

The third component to the statement of Theorem~\ref{thm:1} is the desuspension $\Sigma^{-n}$ of an operad of spectra. For operads of chain complexes, the notion of desuspension is very simple to describe. Given an operad $P$, we define a new operad $s^{-1}P$ by $s^{-1}P(k)_r := P(k)_{r+(k-1)}$, i.e.\ a shift in degree by $k-1$, with $\Sigma_k$-action twisted by the sign representation. The operad composition maps for $s^{-1}P$ are just shifted versions of those of $P$.

The key property of the operad $s^{-1}P$ is that an $s^{-1}P$-algebra can be identified with a $P$-algebra shifted up in degree by $1$. It is this fact that we use to define desuspension of operads of spectra.

Thus, for an operad $\mathbf{P}$ of spectra, we require that its \emph{desuspension}, denoted $\Sigma^{-1}P$, be an operad of spectra for which the ordinary suspension functor for spectra provides an equivalence between the ($\infty$-)categories of $P$-algebras and $\Sigma^{-1}P$-algebras.

Arone and Kankaanrinta provide an explicit construction of such a desuspension in~\cite{arone/kankaanrinta:2014}. There they describe a cooperad $S_\infty$ of pointed spaces with the property that $S_\infty(k)$ is homeomorphic to the sphere $S^{k-1}$, and for which all composition maps are homeomorphisms. (Thus $S_\infty$ is also an operad.) The desuspension of an operad $\mathbf{P}$ of spectra can then be defined via mapping spectra as
\[ \Sigma^{-1}\mathbf{P} := \Map(S_\infty,\mathbf{P}). \]
In the statement of Theorem~\ref{thm:1}, we require an $n$-fold operadic desuspension. In order to obtain a fully $O(n)$-equivariant equivalence, we introduce a coordinate free version of the constructions of \cite{arone/kankaanrinta:2014}, that is a cooperad $S_n$ of pointed spaces that admits an $O(n)$-action and is equivalent to a termwise smash product of $n$ copies of $S_\infty$. (Our notation unfortunately clashes with that of Arone and Kankaanrinta who write $S_n$ for another (co)operad that is equivalent to $S_\infty$.)

\subsection*{The Koszul dual of the inclusion \texorpdfstring{$E_{n} \to E_{m+n}$}{En to E(m+n)}}

Our approach to Theorem~\ref{thm:1} also allows us to identify the Koszul dual of the operad map $E_n \to E_{n+m}$ induced by the standard inclusion $\R^n \to \R^{n+m}$. In section ~\ref{sec:compatibility} we prove the following result.

\begin{theorem} \label{thm:2}
There is a homotopy-commutative diagram of operads
\[ \begin{diagram} \dgARROWLENGTH=1em
  \node{K\mathbf{E}_{n+m}} \arrow{s} \arrow{e,t}{\sim} \node{\Sigma^{-(n+m)}\mathbf{E}_{n+m}} \arrow{s} \\
  \node{K\mathbf{E}_n} \arrow{e,t}{\sim} \node{\Sigma^{-n}\mathbf{E}_n}
\end{diagram} \]
where the left-hand map is Koszul dual to the inclusion $\mathbf{E}_n \to \mathbf{E}_{n+m}$, the horizontal maps are the equivalences of Theorem~\ref{thm:1}, and the right-hand vertical map is the $(n+m)$-fold desuspension of a certain unstable operad map
\[ E_{n+m} \to \Sigma^{m}E_n \]
constructed by May in~\cite{may:1972} and also studied by Ahearn and Kuhn in \cite[\S7]{ahearn/kuhn:2002}.
\end{theorem}

As a consequence we obtain in \ref{cor:lie} a description for the spectral Lie operad as the homotopy inverse limit of the operads $\Sigma^{-n}E_n$. This description also appears in \cite[Prop. 10]{hu/kriz/somberg:2019}.

\subsection*{Applications and future directions}

As mentioned above, it should be possible to find a close connection between the operad equivalence constructed in this paper with the bar and cobar functors between $E_n$-algebras and $E_n$-coalgebras described by Ayala-Francis~\cite{ayala/francis:2021} and Lurie~\cite{lurie:2017}. Steps in this direction are taken by Amabel in~\cite[7.3]{amabel:2019} based on further work of Ayala-Francis~\cite{ayala/francis:2019}.

Another algebraic Koszul duality that has the potential to be realized on the level of spectra is the relationship between the `hypercommutative' and `gravity' operads described by Getzler in~\cite{getzler:1995}. The hypercommutative operad is that given by the homology of the Deligne-Mumford-Knudsen compactifications $\overline{\mathcal{M}}_{0,*+1}$ of the moduli spaces of genus $0$ Riemann surfaces with marked points. Getzler's gravity operad is formed from the (shifted) homology of the uncompactified moduli spaces $\mathcal{M}_{0,*+1}$.

Work of Ward~\cite[Sec. 4]{ward:2019} outlines how the main theorem of this paper may be use to promote Getzler's duality result to an equivalence between operads of spectra
\[ K(\Sigma^\infty_+\overline{\mathcal{M}}_{0,*+1}) \homeq \Sigma \Sigma^\infty_+ \mathcal{M}_{0,*+1} \]
where the operad structure on the right-hand side is identified by Westerland~\cite{westerland:2008} as the homotopy fixed point spectra of the $S^1$-action on the stable little $2$-discs. Drummond-Cole has shown in~\cite{drummondcole:2014} that the topological operad $\overline{\mathcal{M}}_{0,*+1}$ bears a close relationship with the \emph{framed} little $2$-discs operad $fE_2$, and Ward's work on bar construction for non-reduced operads~\cite{ward:2019}, together with Theorem~\ref{thm:1}, ties all these constructions together.

Much of the first author's interest in the present project comes from its connections to Goodwillie calculus. In particular, Greg Arone and the first author showed in~\cite{arone/ching:2016} that the Taylor towers of functors from based spaces to spectra are entirely classified by right modules over the inverse sequence of operads
\[ \dots \to K\mathbf{E}_3 \to K\mathbf{E}_2 \to K\mathbf{E}_1. \]
In light of Theorem~\ref{thm:2}, we can now replace this sequence with one of the form
\[ \dots \to \Sigma^{-3}\mathbf{E}_3 \to \Sigma^{-2}\mathbf{E}_2 \to \Sigma^{-1}\mathbf{E}_1 \]

Arone and the first author showed in~\cite{arone/ching:2017} that the Taylor tower for Waldhausen's algebraic $K$-theory of spaces functor $A(X)$ in particular arises from a module over $K\mathbf{E}_3$ though an explicit description of that module was not given. We might now hope to provide instead the corresponding $\Sigma^{-3}\mathbf{E}_3$-module (or, equivalently, an $\mathbf{E}_3$-module) which would yield formulas for the Taylor tower of the algebraic $K$-theory functor.

\subsection*{Outline of our approach}

Our proof of Theorem~\ref{thm:1} relies on two different models for the little $n$-disc operad $E_n$, which we describe in Sections~\ref{sec:restricted} and \ref{sec:FM} respectively.

The first model is the suboperad $D_n$ of $E_n$ formed by those configurations of discs whose radii sum to $1$, and whose barycenter weighted by the radii is the origin. We prove in Theorem~\ref{thm:D} that the inclusion $D_n \subseteq E_n$ is an equivalence of operads. This should seem plausible: with little discs whose radii sum to $1$ it is still possible to move the discs around each other without moving the barycenter; from knowledge of the homotopy groups of configuration spaces, it is relatively easy to show that $D_n \to E_n$ is surjective on homotopy. In fact, we pursue a more geometric approach here and construct a deformation retraction onto $D_n$ from a larger space that we can easily show to be equivalent to $E_n$.

Our second model for $E_n$ is the familiar \emph{Fulton-MacPherson operad}, which we denote $F_n$. The operad $F_n$ was defined by Getzler and Jones in~\cite{getzler/jones:1994}, and provides a conveniently small model for the topological $E_n$-operad: the space $F_n(k)$ is a manifold with corners of dimension $n(k-1)-1$. The second author proved that $F_n$ is a cofibrant model for $E_n$ in \cite{salvatore:2001}.

Given these two models for $E_n$, our approach to Theorem~\ref{thm:1} is very direct. In Section~\ref{sec:map}, we build explicit maps (of pointed spaces)
\begin{equation} \label{eq:1} F_n(k)_+ \smsh BD_n(k) \to \bar{S}_n(k) \end{equation}
which respect the operad structure on $F_n$ and the bar-cooperad structure on $BD_n$. In (\ref{eq:1}), $\bar{S}_n$ is an operad of pointed spaces that admits an equivalence (i.e.\ an arity-wise weak homotopy equivalence of pointed spaces):
\[ S_n \weq \bar{S}_n \]
from the $n$-sphere operad we use to define the operadic desuspension. Taking suitable adjuncts, we construct from (\ref{eq:1}) a map of operads of spectra
\begin{equation} \label{eq:intro-map} \Sigma^\infty_+ F_n \to \Map(BD_n, \Sigma^\infty \bar{S}_n). \end{equation}
We prove in Section~\ref{sec:proof} that the map (\ref{eq:intro-map}) is an equivalence of operads, which implies Theorem~\ref{thm:1}.

Our proof uses the bar-cobar duality for operads of spectra from~\cite{ching:2012}, recalled in Section~\ref{sec:bar-cobar}, to rewrite (\ref{eq:intro-map}) as a map of quasi-cooperads
\begin{equation} \label{eq:dual-map} \Sigma^\infty \mathbb{B}F_n \to \Map(D_n,\Sigma^\infty \bar{S}_n). \end{equation}
Here $\mathbb{B}$ denotes the left Quillen functor from the bar-cobar Quillen equivalence of \cite{ching:2012}; thus $\mathbb{B}F_n$ is a model for the bar construction $BF_n$. The functor $\mathbb{B}$ is described in more detail in Section~\ref{sec:quasi}.

We ultimately prove that (\ref{eq:dual-map}) is an equivalence using an S-duality result of Dold and Puppe~\cite{dold/puppe:1980} which we recall in Theorem~\ref{thm:dold/puppe}. That result gives conditions under which the one-point compactification $U^+$ of an open subset $U \subseteq \R^N$ should be Spanier-Whitehead $N$-dual to $U$ itself.

We apply the Dold-Puppe result by showing that the space $\mathbb{B}F_n(k)$ is homeomorphic to the one-point compactification of the configuration space of $k$ points in $\R^n$, modulo translation, which we can view as an open subset in $\R^{n(k-1)}$. We conclude that $\mathbb{B}F_n(k)$ is Spanier-Whitehead dual to that configuration space, and hence to $D_n(k)$, up to a shift by dimension $n(k-1)$.

The claim made in the previous paragraph, that the individual terms in the bar construction $\mathbb{B}F_n$ are Spanier-Whitehead duals of the configuration spaces themselves, appeared already in the second author's thesis in 1999. The impediment to proving Theorem~\ref{thm:1} since then has been the difficulty in putting these equivalences together into an actual map of operads, i.e.\ maps that commute on the nose with the operad structures involved. The explicit construction of the maps (\ref{eq:1}), and the demonstration that they have the required properties, is therefore the main substance of this paper.

\subsection*{The barycentre (co)operad}

It is worth highlighting here in the introduction a key object which helped us solve the crucial problem of building Spanier-Whitehead duality maps that respect the operad structures.

As noted above, a standard way to calculate the Spanier-Whitehead dual of a space $U$ is via an embedding of $U$ in Euclidean space $\R^N$ for some $N$. In our case these are the embeddings of the configuration spaces into the Euclidean space $\R^{n(k-1)}$ of all $k$-tuples in $\R^n$, modulo translation. In order that the resulting duality maps preserve the operad structures, we needed these Euclidean spaces to possess an operad structure of their own. But it does not seem to be possible to define suitable composition maps on the spaces $R_n(k) = \R^{nk}/\R^n$ that are strictly associative.

We did consider approaches involving $\infty$-operads, via models such as the dendroidal Segal spaces of Cisinski and Moerdijk~\cite{cisinski/moerdijk:2013}, but in the end we developed a different solution which we refer to as the \emph{barycentre (co)operad}. The key idea is to parametrize
our constructions over the simplex operad of Arone and Kankaanrinta~\cite{arone/kankaanrinta:2014}.

We write
\[ R_n(k) := \{ (x,t) \in (\R^n)^k \times (0,1)^k \; | \; \sum_i t_i = 1, \; \sum_i t_i x_i = 0 \}. \]
The space $R_n(k)$ is a (trivial) vector bundle over the open simplex $\Delta(k) = \Delta^{k-1}$ with fibre isomorphic to the space $\R^{nk}/\R^n$ of $k$-tuples in $\R^n$, modulo translation. The fibre over a given $t \in \Delta(k)$ is the vector space of $k$-tuples that satisfy a barycentre condition with respect to weights in $t$.

The central construction of this paper is an $O(n)$-equivariant (co)operad structure on the spaces $R_n$, described by the formulas in Definition~\ref{def:P2}, which, at least in a fibrewise manner, realizes the goal of an operad of vector spaces into which configuration spaces can be embedded.

Another crucial feature of $R_n$ is the observation that these same formulas underlie the algebra of the little discs operads, where the vectors $x_i$ represent the centres of the little discs, and the numbers $t_i$ represent the radii of those discs. In particular, our restricted little discs operad $D_n$ admits a natural embedding into the operad $R_n$. Finally, the fibrewise one-point compactifications of $R_n$ also form the (co)operads $S_n$ of pointed spaces that we use to model the desuspension of operads of spectra.

\subsection*{A note on notation}

To emphasize the coordinate-free nature of our constructions, we will base all of our $E_n$-operads on an arbitrary finite-dimensional real vector space $V$, writing $E_V$ for the corresponding operad. Similarly, the corresponding Fulton-MacPherson operad will be denoted $F_V$, and so on. This choice also allows us to highlight those constructions that depend on a choice of norm on $V$. The reader that wishes to recover the ordinary little $n$-discs/cubes operad will take $V = \R^n$ with the Euclidean/$\ell^\infty$-norm respectively.

\subsection*{Acknowledgments}

We started this project over 15 years ago, but the key insights came during our stay at the Isaac Newton Institute in Cambridge, during the programme \emph{Homotopy Harnessing Higher Structures} in 2018. We would like to thank the Isaac Newton Institute, and the organizers of the HHH programme, for their support and hospitality. We are also particularly grateful to Benoit Fresse for many discussions on this project, and for inviting us to visit Universit\'{e} de Lille where some of the ideas in this paper were developed. The first author would like to thank Greg Arone for continual conversation and encouragement over those 15 years. We are also grateful to an anonymous referee for careful reading and feedback which led to several improvements, in particular to Section~\ref{sec:restricted}.

\section{The barycentric, simplex and sphere (co)operads}

In this section we will describe the framework we use for operads and cooperads of topological spaces, and introduce some basic examples that underlie many of the constructions of this paper.

The traditional definition of an operad of spaces starts with a sequence $(P(n))_{n \geq 0}$ of spaces together with an action of the symmetric group $\Sigma_n$ on $P(n)$. The operad structure maps consist of `composition' maps of the form
\[ P(k) \times P(n_1) \times \dots \times P(n_k) \to P(n_1+\dots+n_k) \]
and a `unit' map
\[ * \to P(1) \]
that together satisfy standard associativity and unit conditions.

All the operads and cooperads we wish to consider in this paper are \emph{reduced} in the sense that they satisfy $P(0) = \varnothing$ and $P(1) = *$. We will build these properties into our definition of operad by only considering the values $P(n)$ for $n \geq 2$. Given this restriction, the data of an operad can instead be described via `partial' composition maps of the form
\[ \circ_i: P(k) \times P(l) \to P(k+l-1) \]
for $i = 1,\dots,k$.

In order to easily include the symmetric group equivariance, and avoid the complication of renumbering issues, we think of the underlying sequence of spaces in the operad as a functor on the category $\mathsf{FinSet}_{\geq 2}$ of finite sets of cardinality at least $2$ and bijections. The partial composition maps can then be described in the form
\[ \circ_i : P(I) \times P(J) \to P(I \cup_i J) \]
for $i \in I$, where $I$ and $J$ are finite sets (of cardinality at least $2$) and $I \cup_i J$ denotes the disjoint union $(I - \{i\}) \amalg J$. These maps should be natural with respect to bijections $I \isom I'$ and $J \isom J'$, as well as satisfying the usual associativity conditions. We therefore arrive at the following definition.

\begin{definition} \label{def:operad}
A \emph{(reduced) operad} of topological spaces consists of:
\begin{itemize}
  \item a functor $P: \mathsf{FinSet}_{\geq 2} \to \spaces$, where $\spaces$ is the category of topological spaces;
  \item for $I,J \in \mathsf{FinSet}_{\geq 2}$ and each $i \in I$, a \emph{composition map}
  \[ \circ_i: P(I) \times P(J) \to P(I \cup_i J). \]
\end{itemize}
The composition maps $\circ_i$ should satisfy the following conditions:
\begin{itemize}
  \item naturality in $I$ and $J$;
  \item two forms of associativity: for $I,J,K \in \mathsf{FinSet}_{\geq 2}$, $i \in I$ and $j \in J$:
  \[ \begin{diagram}
    \node{P(I) \times P(J) \times P(K)} \arrow{e,t}{\circ_i \times P(K)} \arrow{s,l}{P(I) \times \circ_j} \node{P(I \cup_i J) \times P(K)} \arrow{s,r}{\circ_j} \\
    \node{P(I) \times P(J \cup_j K)} \arrow{e,t}{\circ_i} \node{P(I \cup_i J \cup_j K)}
  \end{diagram} \]
  and for $I,J,J' \in \mathsf{FinSet}_{\geq 2}$, $i \neq i' \in I$:
  \[ \begin{diagram}
    \node{P(I) \times P(J) \times P(J')} \arrow{e,t}{\circ_i \times P(J')} \arrow{s,l}{(\circ_{i'} \times P(J))(P(I) \times \sigma)} \node{P(I \cup_i J) \times P(J')} \arrow{s,r}{\circ_{i'}} \\
    \node{P(I \cup_{i'} J') \times P(J)} \arrow{e,t}{\circ_i} \node{P(I \cup_i J \cup_{i'} J')}
  \end{diagram} \]
\end{itemize}
where $\sigma$ interchanges the factors $P(J)$ and $P(J')$.
\end{definition}

\begin{definition} \label{def:cooperad}
A \emph{cooperad} $Q$ of topological spaces is an operad in the opposite category of topological spaces. In particular, a cooperad consists of spaces $Q(I)$ together with \emph{decomposition maps}
\[ \nu_i: Q(I \cup_i J) \to Q(I) \times Q(J) \]
satisfying diagrams dual to those in Definition~\ref{def:operad}. Note that if the composition maps for an operad are homeomorphisms then their inverses are the decomposition maps for a cooperad, and vice versa.
\end{definition}

\begin{definition} \label{def:operad-spectra}
Let $\spectra$ be a symmetric monoidal model for stable homotopy theory, for example, the category of symmetric spectra. We define \emph{operads} or \emph{cooperads of spectra} as in Definitions~\ref{def:operad} and~\ref{def:cooperad} with the category of topological spaces replaced by $\spectra$ and the cartesian product replaced by the smash product of spectra.
\end{definition}

\begin{example}
Assume that there is a strong monoidal model for the suspension spectrum functor $\Sigma^\infty_+: \spaces \to \spectra$. Then for any operad (or cooperad) $P$ of topological spaces, the termwise suspension spectrum $\Sigma^\infty_+P$ is an operad (respectively, a cooperad) of spectra.
\end{example}

Almost all of the operads in this paper are built from the following example which we call the \emph{overlapping discs operad}.

\begin{definition} \label{def:P}
Let $V$ be a finite-dimensional real vector space, and let $I$ be a nonempty finite set. Let $P_V(I)$ be the topological space
\[ P_V(I) := V^I \times (0,\infty)^I. \]
By choosing a norm on $V$, we might visualize a point $(x,t)$ in $P_V(I)$ as a collection of (closed) discs in $V$, indexed by the set $I$, where the $i$-th disc, for $i \in I$, has centre $x_i \in V$ and radius $t_i > 0$. There are no conditions preventing the discs from overlapping. When $V = 0$, a point in $P_0(I)$ can be identified simply with an $I$-indexed sequence $t$ of positive real numbers.
\end{definition}

We define composition maps that make the symmetric sequence $P_V$ into an operad of topological spaces by extending the structure of the usual little discs operad to the spaces $P_V(I)$.

\begin{definition} \label{def:P2}
For nonempty finite sets $I,J$ and $i \in I$, we define
\[ +_i: P_V(I) \times P_V(J) \to P_V(I \cup_i J) \]
by
\[ (x,t), (y,u) \mapsto (x +_i ty, t \cdot_i u) \]
where
\[ (t \cdot_i u)_j := \begin{cases} t_j & \text{if $j \notin J$}; \\ t_iu_j & \text{if $j \in J$}; \end{cases} \]
and
\[ (x +_i ty)_j := \begin{cases} x_j & \text{if $j \notin J$}; \\ x_i + t_iy_j & \text{if $j \in J$}. \end{cases} \]
\end{definition}

\begin{proposition} \label{prop:P}
The composition maps Definition~\ref{def:P2} make $P_V$ into an operad of topological spaces, which we refer to as the \emph{overlapping $V$-discs operad}.
\end{proposition}

The operad structure on $P_V$ can be visualized in the much the same way as for the standard little discs operad, except without any of the usual restrictions on the positions and sizes of the `little' discs. Starting with two configurations of discs $(x,t) \in P_V(I)$ and $(y,u) \in P_V(J)$, the configuration $(x,t) +_i (y,u)$ is given by dilating the configuration $(y,u)$ by a factor of $t_i$, and inserting it in place of the $i$-th disc of the configuration $(x,t)$, i.e.\ centred at the point $x_i \in V$.

The familiar little-discs operad can be identified as a suboperad of $P_V$, which we describe via a choice of norm on the vector space $V$.

\begin{proposition} \label{prop:E}
Let $V$ be finite-dimensional real normed vector space. Then the subspaces
\[ E_V(I) := \left\{ (x,t) \in P_V(I) \; | \; |x_i| \leq 1-t_i, \; |x_i - x_j| \geq t_i+t_j \right\} \]
form a suboperad of $P_V$, which we refer to as the \emph{little $V$-disc operad}.
\end{proposition}

\begin{remark}
When $V$ is $\R^n$ with the Euclidean norm, the operad $E_V$ is the ordinary little $n$-discs operad of Boardman-Vogt~\cite{boardman/vogt:1973}, in which a point consists of a collection of non-overlapping discs inside the unit disc in $\R^n$. Taking the $\ell^\infty$-norm on $\R^n$ instead, we obtain the little $n$-\emph{cubes} operad.
\end{remark}

\begin{definition}
For a (topological) group $G$, a \emph{$G$-operad} consists of an operad $P$ together with an action of $G$ on each $P(I)$ such that all the structure maps for $P$ are $G$-equivariant.
\end{definition}

\begin{example}
The general linear group $GL(V)$ acts on each space $P_V(I)$ via the diagonal action on $V^I$ and trivially on $(0,\infty)$. Since the composition maps for $P_V$ are linear in the vector space components, these actions make $P_V$ into a $GL(V)$-operad. The orthogonal group $O(V)$, for the normed vector space $V$, acts on $E_V(I)$, and makes $E_V$ into an $O(V)$-operad.
\end{example}

We will make particular use of the following suboperad of $P_V$.

\begin{definition} \label{def:R}
Let $R_V(I)$ be the subspace of $P_V(I)$ consisting of the following spaces
\[ \textstyle R_V(I) := \{ (x,t) \in P_V(I) \; | \; \sum_{i \in I} t_i = 1, \; \sum_{i \in I} t_i x_i = 0 \}. \]
In other words, a configuration of discs is in $R_V(I)$ if the sum of the radii of the discs is $1$, and the centres of the discs have (weighted) barycentre equal to the origin.
\end{definition}

\begin{proposition} \label{prop:R}
The subspaces $R_V(I)$ form a $GL(V)$-suboperad of $P_V$ for which the composition maps are homeomorphisms. We will refer to $R_V$ as the \emph{barycentre (co)operad}.
\end{proposition}
\begin{proof}
It is easy to check that the conditions on $R_V(I)$ are stable under the composition maps in Definition~\ref{def:P2}. To see that these composition maps are homeomorphisms on $R_V$ we explicitly describe the inverses.

For $(x,t) \in R_V(I \cup_i J)$ we define $(x/J,t/J) \in R_V(I)$ by
\[ \begin{split} (x/J)_{i'} &:= x_{i'} \quad \text{if $i' \neq i$}; \\ (x/J)_i &:= \frac{\sum_{j \in J} t_j x_j}{\sum_{j \in J} t_j} \end{split} \]
and
\[ \begin{split} (t/J)_{i'} &:= t_{i'} \quad \text{if $i' \neq i$}; \\ (t/J)_i &:= \sum_{j \in J} t_j \end{split}. \]
We also define $(x|J,t|J) \in R_V(J)$ by
\[ (x|J)_{j} := \frac{x_{j} - \frac{\sum_{j \in J} t_j x_j}{\sum_{j \in J} t_j}}{\sum_{j' \in J} t_{j'}} \]
and
\[ (t|J)_{j} := \frac{t_{j}}{\sum_{j' \in J} t_{j'}}. \]
It is simple to check that $(x,t) \mapsto (x/J,t/J),(x|J,t|J)$ defines a continuous inverse to the composition map $R_V(I) \times R_V(J) \to R_V(I \cup_i J)$, making $R_V$ also into a cooperad.
\end{proof}

Setting $V = 0$ in Definition~\ref{def:R} gives a (co)operad of special importance.

\begin{definition} \label{def:Delta}
Let $\Delta = R_0$ be the barycentre (co)operad corresponding to the zero vector space. Then we have
\[ \textstyle \Delta(I) = \{ t \in (0,\infty)^I \; | \; \sum_{i \in I} t_i = 1 \} \]
with composition maps (homeomorphisms) $t,u \mapsto t \cdot_i u$ given by the formula in Definition~\ref{def:P}. Since the terms in the (co)operad $\Delta$ are the ordinary (open) topological simplexes, we also refer to $\Delta$ as the \emph{simplex (co)operad}. This operad was also described by Arone and Kankaanrinta~\cite{arone/kankaanrinta:2014} where it is denoted $\Delta_1$.
\end{definition}

\begin{proposition} \label{prop:R-bundle}
The projection map
\[ R_V(I) \to \Delta(I); \quad (x,t) \mapsto t \]
is a trivial vector bundle with fibre isomorphic to the vector space $V^I/V$ of $I$-tuples in $V$ modulo translation. Together these maps form a morphism of operads $R_V \to \Delta$.
\end{proposition}
\begin{proof}
The necessary homeomorphism $R_V(I) \isom V^I/V \times \Delta(I)$ is given by
\[ (x,t) \mapsto ([x],t) \]
where $[x]$ denotes the equivalence class of $x \in V^I$ modulo diagonal 
translation. The inverse map picks out the unique representative in the equivalence class $[x]$ that has weighted barycentre $0$.
\end{proof}

\begin{definition} \label{def:SV}
Let $S_V(I)$ denote the Thom space of the bundle $R_V(I) \to \Delta(I)$ of Proposition~\ref{prop:R-bundle}. In other words, $S_V(I)$ is obtained from the fibrewise one-point compactification of $R_V(I)$ over $\Delta(I)$,
by identifying all the points at infinity. Since $R_V(I)$ is a trivial bundle with fibre $V^I/V$, we can identify $S_V(I)$ with $S^{V^I/V} \smsh \Delta(I)_+$. Note that $S_V(I)$ is homotopy equivalent to a sphere of dimension $(|I|-1)\dim(V)$.
\end{definition}

\begin{proposition} \label{prop:S}
The operad structure maps of $R_V$ extend to homeomorphisms
\[ S_V(I) \smsh S_V(J) \isom S_V(I \cup_i J) \]
which make $S_V$ into a $GL(V)$-(co)operad of pointed spaces, which we refer to as the \emph{$V$-sphere-(co)operad}.
\end{proposition}
\begin{proof}
These are the homeomorphisms of Thom spaces induced by the isomorphism of vector bundles (over $\Delta(I) \times \Delta(J) \isom \Delta(I \cup_i J)$) of the form
\[ R_V(I) \times R_V(J) \isom R_V(I \cup_i J); \quad (x,t),(y,u) \mapsto (x+_i ty,t \cdot_i u). \qedhere \]
\end{proof}

\begin{remark}
The $V$-sphere (co)operad $S_V$ is our analogue of the `sphere (co)operads' of Arone-Kankaanrinta~\cite{arone/kankaanrinta:2014}. The objects described there are operads/cooperads whose terms are actual spheres unlike our more complicated fibrewise constructions. The advantages in this paper of the fibrewise version $S_V$ are its closer connection to the structure of the little $V$-disc operad and the $GL(V)$-action inherited from the action on $V$. We will use $S_V$ to define our version of `operadic suspension' in Section~\ref{sec:susp} below, where we discuss the relationship with the work of Arone and Kankaanrinta in more detail.
\end{remark}

\begin{notation}
It will be convenient to write
\[ x_J := \frac{\sum_{j \in J} t_j x_j}{\sum_{j \in J} t_j} \]
for the \emph{weighted barycentre} of a collection of vectors $(x_j)_{j \in J}$ with respect to a sequence of \emph{weights} $(t_j)_{j \in J}$, and
\[ t_J := \sum_{j \in J} t_j \]
for the \emph{combined weight} of such a sequence.
\end{notation}

We close this section by providing some useful formulas that tell us how the operations of taking weighted barycentres and combined weights interact with the structure maps in the barycentre and simplex quasi-operads.

\begin{lemma} \label{lem:btcalc}
Consider a point $(x,t) \in R_V(I \cup_i J)$. For each subset $K \subseteq I$ that does not contain $i$, we have
\[ (x/J)_K = x_K, \quad (t/J)_{K} = t_K, \]
and for each subset $K \subseteq I$ that does contain $i$, we have
\[ (x/J)_K = x_{K \cup_i J}, \quad (t/J)_{K} = t_{K \cup_i J}. \]
For each subset $K \subseteq J$, we have
\[ (x|J)_K = t_J^{-1}(x_K - x_J), \quad (t|J)_K = t_J^{-1}t_K. \]
\end{lemma}
\begin{proof}
The results in the first line follow immediately from the definitions of $x/J$ and $t/J$. So suppose $i \in K \subseteq I$. Then we have
\[ (t/J)_K = \sum_{j \in K - \{i\}} t_j + (t/J)_i = \sum_{j \in K - \{i\}} t_j + t_J = \sum_{j \in K \cup_i J} t_j = t_{K \cup_i J} \]
and
\[ (x/J)_K = \frac{\sum_{j \in K - \{i\}} t_jx_j + t_Jx_J}{(t/J)_K} = \frac{\sum_{j \in K - \{i\}} t_jx_j + \sum_{j \in J} t_jx_j}{t_{K \cup_i J}} = x_{K \cup_i J}. \]
For $K \subseteq J$, we have
\[ (t|J)_K = \sum_{j \in K} (t|J)_j = \sum_{j \in K} t_J^{-1}t_j = t_J^{-1}t_K \]
and
\[ \begin{split} (x|J)_K &= \frac{\sum_{j \in K} (t|J)_j(x|J)_j}{(t|J)_K} = \frac{\sum_{j \in K} t_J^{-1}t_jt_J^{-1}(x_j - x_J)}{t_J^{-1}t_K} \\ &= \frac{t_J^{-1}((\sum_{j \in K} t_jx_j) - t_Kx_J)}{t_K} = t_J^{-1}(x_K - x_J). \end{split} \qedhere \]
\end{proof}

\section{The restricted little-disc operad} \label{sec:restricted}

A key component of the construction of our duality map will be the ability to embed an $E_n$-operad into the barycentre operad $R_V$ (where $n = \dim V$). For this purpose, we now construct a version of the little discs operad $E_V$ that includes the barycentre and simplex conditions we used to define the barycentre (co)operad $R_V$.

\begin{definition} \label{def:D}
Let $V$ be a finite-dimensional real normed vector space. The \emph{restricted little $V$-disc operad} is the suboperad $D_V$ of $E_V$ given by
\[ D_V(I) := R_V(I) \cap E_V(I) = \{ (x,t) \in E_V(I) \; | \; \sum_i t_i = 1, \; \sum_i t_ix_i = 0 \}. \]
In other words, $D_V$ is the operad consisting of those collections of non-overlapping discs inside the unit disc where the sum of the radii of the discs is equal to $1$, and the centres of the discs have weighted barycentre the origin. (It is straightforward to check that these conditions are preserved under the composition maps for the little disc operad $E_V$; that is the content of Proposition~\ref{prop:R}.)
\end{definition}

\begin{remark}
Since $D_V$ is also a suboperad of the (co)operad $R_V$, its composition maps are necessarily injective. In fact, the composition map
\[ D_V(I) \times D_V(J) \to D_V(I \cup_i J) \]
is the embedding of a closed subspace. (We will prove a more general embedding statement in Lemma~\ref{lem:DV}.)
\end{remark}

The main result of this section is that this restricted version of the little disc operad retains the same homotopy type.

\begin{theorem} \label{thm:D}
For a finite-dimensional real normed vector space $V$, the inclusion
\[ D_V \into E_V \]
is an equivalence of operads.
\end{theorem}
\begin{proof}
It is well-known, for example by \cite{may:1972}, that, for each finite set $I$, there is a homotopy equivalence
\[ p: E_V(I) \to C_V(I); \quad (x,t) \mapsto x, \]
where $C_V(I)$ is the configuration space of $I$-tuples of distinct points in $V$. It is therefore sufficient to show that the restriction of $p$ to $D_V(I)$ is also an equivalence, which we show in the following proposition.
\end{proof}

\begin{proposition} \label{prop:D}
For each finite set $I$, the map $p: D_V(I) \to C_V(I)$, given by $p(x,t) = x$, is an equivalence. Moreover, for each $t \in \Delta(I)$, the restriction of $p$ to $D_V(I)_t$, that is, the fibre over $t$ of the projection $D_V(I) \to \Delta(I)$, is also an equivalence.
\end{proposition}
\begin{proof}
Consider the following subspace of the barycentre space $R_V(I)$:
\[ RD_V(I) := \{ (x,t) \in R_V(I) \; | \; |x_i - x_j| \geq t_i+t_j \}. \]
Thus $RD_V(I)$ is similar to the restricted little $V$-disc space $D_V(I)$ but without the bounding constraints $|x_i| \leq 1 - t_i$. In other words, $RD_V(I)$ is the space of $I$-tuples of discs in $V$ with disjoint interiors, whose radii sum to $1$, and whose centres satisfy the weighted barycentre condition.

We first claim that the forgetful map $p: RD_V(I) \to C_V(I); (x,t) \mapsto x$ has the following homotopy inverse. Given $x \in C_V(I)$, and $i \in I$, we set
\[ u_i := \frac{1}{2}\min_{i' \neq i}\{|x_{i'} - x_i| \}, \]
and note that $u \in (0,\infty)^I$ depends continuously on $x$. Then define $C_V(I) \to RD_V(I)$ by
\[ x \mapsto (u_I^{-1}(x - x_I), u_I^{-1} u). \]
It is straightforward to check that this is the desired homotopy inverse for $p$. Note that a similar argument shows that the restriction of $p$ to the fibre $RD_V(I)_t$ is also an equivalence.

It is now sufficient to show that the inclusion
\[ D_V(I) \into RD_V(I) \]
is a (fibrewise over $\Delta(I)$) homotopy equivalence. The required result follows from Proposition~\ref{prop:defret} below, but it will take us some time to get there.
\end{proof}

Let $r: R_V(I) \to (0,\infty)$ be the continuous function defined by
\begin{equation} \label{eq:r} r(x,t) := \max_{i \in I} \{|x_i|+t_i\} \end{equation}
We write $R_V(I)^{>1}$ for the set of points in $R_V(I)$ for which $r(x,t) > 1$. Notice that by definition $D_V(I) = RD_V(I) - R_V(I)^{>1}$.

We will produce a deformation retraction of $RD_V(I)$ onto $D_V(I)$ by constructing a smooth vector field $X$ on $R_V(I)^{>1}$ such that $r$ decreases along each integral curve of $X$ at some rate which we can control. The flow for the vector field $X$, restricted to $RD_V(I)$, will then provide the desired deformation.

The basic idea for this flow is as follows. Take $(x,t) \in R_V(I)^{>1}$, representing some configuration of discs in $V$ with weighted barycentre the origin. We move each connected component of this configuration rigidly towards the origin. When two or more components touch, they `stick' together and the resulting component is then moved rigidly. We will show that moving along this flow strictly decreases the function $r$ of (\ref{eq:r}) at a suitable rate.

Unfortunately, the flow described in the previous paragraph is not continuous. We therefore use a partition of unity argument to smooth out the underlying vector field. Intuitively, this means that two connected components of a given configuration start to influence one another once they get close, and before they actually touch.

Our approach is inspired by that of Baryshnikov, Bubenik and Kahle in \cite{baryshnikov/bubenik/kahle:2014} who prove, among other things, that the space of $k$ little discs of radius $\rho$ inside the unit disc has the same homotopy type as the configuration space when $\rho < 1/k$. Our situation does not fit their exact framework because: (i) we want to consider the limiting case where the sum of the radii of the discs equals $1$; (ii) we want to treat discs of varying radii, which makes the spaces we are working with noncompact; and (iii) we wish to include cases such as for the little cubes version of our result, where the norm involved, the $\ell^\infty$-norm, is not smooth. Our approach is therefore more ad hoc though in spirit follows section 3 of \cite{baryshnikov/bubenik/kahle:2014}.

We start by introducing notation for the connected components of a configuration of discs, and we prove a key bound that is central to our calculations.

\begin{definition}
We say that a subset $J \subseteq I$ is \emph{$(x,t)$-connected} if for any $j,j' \in J$ there is a sequence
\[ j = j_0, \dots, j_r = j' \]
in $J$ such that $|x_{j_u} - x_{j_{u+1}}| \leq t_{j_u} + t_{j_{u+1}}$ for $u = 0,\dots,r-1$. In other words, the union of the discs in the configuration $(x,t)$ that correspond to elements of $J$ is a connected subset of $V$. The maximal $(x,t)$-connected subsets of $I$ form a partition which we denote $\pi(x,t)$.
\end{definition}

\begin{lemma} \label{lem:part}
Take $(x,t) \in R_V(I)$, let $J$ be an $(x,t)$-connected set, and suppose $j \in J$. Then we have
\[ |x_j - x_{J}| \leq t_{J} - t_j. \]
In particular, if $\pi(x,t)$ is the indiscrete partition $\{I\}$, then $r(x,t) \leq 1$.
\end{lemma}
\begin{proof}
We use induction on the number of elements of $J$. If $J$ has a single element, the claim is immediate. For the induction step, suppose $J$ has more then one element, and consider the maximal $(x,t)$-connected subsets of $J - \{j\}$ which we denote $J_1,\dots,J_k$.

For each such subset $J_i$, there is some $j_i \in J_i$ such that
\[ |x_j - x_{j_i}| \leq t_j + t_{j_i}. \]
Applying the induction hypothesis to $J_i$ tells us that
\[ |x_{j_i} - x_{J_i}| \leq t_{J_i} - t_{j_i} \]
and so
\[ |x_j - x_{J_i}| \leq t_{J_i} + t_j. \]
The point $x_{J}$ is the weighted barycentre of $x_j,x_{J_1},\dots,x_{J_k}$ and so we have
\[ t_J(x_j - x_J) = t_J x_j - t_j x_j - t_{J_1}x_{J_1} - \dots - t_{J_k}x_{J_k} = t_{J_1}(x_j - x_{J_1}) + \dots + t_{J_k}(x_j - x_{J_k}). \]
Therefore
\[ |x_j - x_J| \leq \frac{t_{J_1}(t_{J_1}+t_j) + \dots + t_{J_k}(t_{J_k}+t_j)}{t_J} \leq \frac{t_{J_1}t_J + \dots + t_{J_k}t_J}{t_J} = t_J - t_j \]
as required.

The final claim follows from observing that $x_I = 0$ and $t_I = 1$, so that the lemma gives $|x_i| \leq 1-t_i$.
\end{proof}

We now define a vector field whose flow moves each connected component in a configuration of discs rigidly towards the origin.

\begin{definition}
For each partition $\pi$ of $I$, we define a smooth vector field $X^{\pi}$ on $R_V(I)$ by
\[ X^{\pi}_i := -x_{[i]_{\pi}} \]
where $[i]_{\pi}$ is the piece of the partition $\pi$ that contains $i \in I$, and with the components of $X^{\pi}$ in the $t$-direction all equal to $0$. We have
\[ \sum_{i \in I} t_i X^{\pi}_i = \sum_{i \in I} -t_i x_{[i]_{\pi}} = \sum_{i \in I} \sum_{j \in [i]_{\pi}} -t_i \frac{t_j x_j}{t_{[i]_{\pi}}} = -\sum_{j \in I} t_j x_j \sum_{i \in [j]_{\pi}} \frac{t_i}{t_{[j]_{\pi}}} = -\sum_{j \in I} t_j x_j = 0, \]
by the barycentre condition, since for $i \in [j]_{\pi}$ we have $[i]_{\pi} = [j]_{\pi}$. Therefore $X^{\pi}$ is indeed a vector field on $R_V(I)$.
\end{definition}

To splice the vector fields $X^{\pi}$ together, we require a suitable partition of unity.

\begin{definition} \label{def:Upi}
For each partition $\pi$ of $I$, let $U_{\pi}$ be the open subset of $R_V(I)^{>1}$ consisting of those points $(x,t)$ such that
\begin{itemize}
    \item the partition $\pi(x,t)$ is a refinement of (or equal to) $\pi$, i.e.\ we have $[i]_{\pi(x,t)} \subseteq [i]_{\pi}$ for every $i$, and
    \item $\displaystyle |x_i - x_{[i]_{\pi}}| < \frac{r(x,t)+1}{2} - \hat{t} - t_i$ for all $i \in I$.
\end{itemize}
where $\hat{t} = \min_{i \in I}\{t_i\}$.

The idea is that $U_\pi$ consists of those configurations of discs that are `close' to being connected in the pattern of the partition $\pi$. In particular, the second condition implies that the discs corresponding to one piece of that partition cannot be too widely separated relative to the size of the whole configuration. The reason for the specific formula appearing above will become apparent in the proof of Lemma~\ref{lem:curve}.
\end{definition}

\begin{lemma} \label{lem:Upi}
For $(x,t) \in R_V(I)^{>1}$, if $\pi(x,t) = \pi$, then $(x,t) \in U_{\pi}$. Therefore
\[ \mathcal{U} = \{ U_{\pi} \; | \; \pi \neq \{I\} \} \]
is an open cover of $R_V(I)^{>1}$.
\end{lemma}
\begin{proof}
Since $r(x,t) > 1$, we have $\pi(x,t) \neq \{I\}$ by Lemma~\ref{lem:part}. For each $i \in I$, we therefore have $[i]_{\pi} \neq I$, so that $t_{[i]_{\pi}} \leq 1 - \hat{t}$. Also, $[i]_{\pi}$ is $(x,t)$-connected, so Lemma~\ref{lem:part} also gives
\[ |x_i - x_{[i]_{\pi}}| \leq t_{[i]_{\pi}} - t_i \leq (1 - \hat{t}) - t_i < \frac{r(x,t)+1}{2} - \hat{t} - t_i. \]
Thus $(x,t) \in U_{\pi}$.
\end{proof}

\begin{definition}
Let $\{\rho_{\pi} \; | \; \pi \neq \{I\}\}$ be a smooth partition of unity subordinate to the open cover $\mathcal{U}$, so that $\mathrm{supp}(\rho_{\pi}) \subseteq U_{\pi}$ for all $\pi$. Define a smooth vector field $X$ on $R_V(I)^{>1}$ by
\[ X = \sum_{\pi} \rho_{\pi}X^{\pi} \]
i.e.\ $X_i = -\sum \rho_{\pi} x_{[i]_{\pi}}$, with the $t$-components of $X$ all equal to $0$.
\end{definition}

We now claim that the integral curves for the vector field $X$ extend in the following way.

\begin{lemma} \label{lem:curve}
Let $\gamma: [0,b) \to R_V(I)^{>1}$ be a maximal integral curve for $X$ on $R_V(I)^{>1}$. Then $b < \infty$, and $\gamma$ extends continuously to a curve
\[ \gamma: [0,\infty) \to R_V(I) \]
by setting
\[ \gamma(s') = \lim_{s \to b} \gamma(s) \]
for all $s' \geq b$. We then also have $r(\gamma(s')) = 1$ for all $s' \geq b$.
\end{lemma}
\begin{proof}
First note that the parameter $t \in \Delta(I)$ is constant along the curve $\gamma$ since $X$ has zero components in that direction. We now claim that for all $s \in [0,b)$ we have, for sufficiently small $h \geq 0$:
\[ r(\gamma(s+h)) \leq r(\gamma(s)) - h\hat{t}. \]
In other words, we prove that the rate at which the function $r$ decreases along the curve $\gamma$ is at least the quantity $\hat{t} = \min_{i \in I}\{t_i\} > 0$.

To see this claim, write $\gamma(s) = (x,t)$, and suppose first that $i \in I$ is such that $r(\gamma(s)) = r(x,t) = |x_i|+t_i$.

The curve $\gamma$ is differentiable at $(x,t)$ with derivative $X(x,t)$, and so by Taylor's Theorem, we have
\[ \gamma(s+h)_i = x_i + hX_i(x,t) + hY(h) \]
where $\lim_{h \to 0} Y(h) = 0$. Therefore, for $0 < h < 1$:
\[ \begin{split} |\gamma(s+h)_i| &= |x_i - \sum_{\pi} \rho_{\pi}hx_{[i]_{\pi}} + hY(h)| \\
    &= | \sum_{\pi} \rho_{\pi} h(x_i - x_{[i]_{\pi}}) + (1-h)x_i + hY(h)| \\
    &\leq \sum_{\pi} \rho_{\pi} h|x_i - x_{[i]_{\pi}}| + (1-h)|x_i| + h|Y(h)| \\
\end{split} \]
Since $\rho_{\pi}$ is supported on $U_{\pi}$, each non-zero term in the sum satisfies
\[ |x_i - x_{[i]_{\pi}}| < \frac{r(x,t)+1}{2} - \hat{t} - t_i. \]
Therefore:
\[ \begin{split}
    |\gamma(s+h)_i| &\leq \sum_{\pi} \rho_{\pi} h\left(\frac{r(x,t)+1}{2} - \hat{t} - t_i \right) + (1-h)|x_i| + h|Y(h)| \\
    &= |x_i| + h\left(\frac{r(x,t)+1}{2} - \hat{t} - t_i - |x_i| + |Y(h)| \right) \\
    &= |x_i| - h\left(\frac{r(x,t)-1}{2} +\hat{t} - |Y(h)| \right).
\end{split} \]
Since $\lim_{h \to 0} |Y(h)| = 0$, and $r(x,t) > 1$, we have, for sufficiently small $h > 0$:
\[ |\gamma(s+h)_i| \leq |x_i| - h\hat{t} = r(\gamma(s)) - t_i - h\hat{t}. \]
Now suppose $i \in I$ is such that $|x_i| + t_i < r(x,t)$. Then, by continuity, we also have, for sufficiently small $h > 0$: $|\gamma(s+h)_i| + t_i \leq r(x,t) - h\hat{t}$ and so also
\[ |\gamma(s+h)_i| \leq r(\gamma(s)) - t_i - h\hat{t}. \]
Putting these inequalities together for all $i \in I$, we get, for sufficiently small $h > 0$:
\[ r(\gamma(s+h)) \leq r(\gamma(s)) - h\hat{t} \]
as required.

Next we deduce that
\[ r(\gamma(s)) \leq r(\gamma(0)) - s\hat{t} \]
for all $s \in [0,b)$: consider the subset of $[0,b)$ consisting of those $s'$ such that the above inequality holds for all $0 \leq s \leq s'$. That subset is closed by continuity of $r$ and $\gamma$, open by the previous result, and nonempty because it contains $0$, so must equal $[0,b)$.

It now follows that $b \leq \frac{r(\gamma(0)) - 1}{\hat{t}}$, so $b$ is finite as claimed.

Since the derivative of $\gamma$ is bounded:
\[ |\gamma'_i| = |X_i| \leq \sum_{\pi} \rho_{\pi} |x_{[i]_{\pi}}| \leq \max_{j \in I}\{|x_j|\} \leq r(\gamma(0)), \]
we deduce that the limit
\[ \gamma(b) := \lim_{s \to b} \gamma(s) \]
exists in $R_V(I)$. Finally, if we had $r(\gamma(b)) > 1$, then the integral curve $\gamma: [0,b) \to R_V(I)^{>1}$ would not be maximal, so $r(\gamma(b)) = 1$. Therefore, defining $\gamma(s') = \gamma(b)$ for all $s' \geq b$ satisfies the required conditions.
\end{proof}

\begin{definition}
Define $\Phi: R_V(I) \times [0,\infty) \to R_V(I)$ as follows:
\begin{itemize}
    \item If $r(x,t) > 1$, then let $\Phi((x,t),-)$ be the integral curve for $X$, extended as in Lemma~\ref{lem:curve}, that starts at $(x,t)$.
    \item If $r(x,t) \leq 1$, we take $\Phi((x,t),s) = (x,t)$ for all $s \in [0,\infty)$.
\end{itemize}
\end{definition}

\begin{proposition}
The function $\Phi$ is continuous and restricts to a function
\[ RD_V(I) \times [0,\infty) \to RD_V(I). \]
\end{proposition}
\begin{proof}
To show that $\Phi$ is continuous, suppose $\Phi((x,t),s) = (y,t)$. If $r(y,t) > 1$, then also $r(x,t) > 1$, and continuity at $((x,t),s)$ follows from that of the maximal flow associated to the smooth vector field $X$. If $r(y,t) < 1$, then $x = y$ and continuity is clear. So suppose $r(y,t) = 1$.

Take a neighbourhood of $(y,t)$ which we may assume contains an open set of the form
\[ Y_{\epsilon} := \{ (y',t') \in R_V(I) \; | \; |y'_i - y_i| < \epsilon, \quad |t'_i - t_i| < \epsilon, \text{ for all $i \in I$} \}. \]
for some $\epsilon$ with $0 < \epsilon < \hat{t}$. Our goal is to find a neighbourhood $U$ of $((x,t),s)$ such that $\Phi(U) \subseteq Y_{\epsilon}$.

\underline{Case 1}: $r(x,t) > 1$. 

Let us write $R_V(I)^{(a,b)}$ for set of points $(y',t') \in R_V(I)$ such that $a < r(y',t') < b$. Since the integral curve starting at $(x,t)$ is continuous and passes through the point $(y,t)$, there is some $s_0 < s$ such that
\[ \Phi((x,t),s_0) \in Y_{\epsilon/2} \cap R_V(I)^{(1,1+\epsilon^2/8)}. \]
By continuity of the flow on $R_V(I)^{>1}$ associated to the smooth vector field $X$, we can find a neighbourhood $U'$ of $((x,t),s_0)$ such that $\Phi(U') \subseteq Y_{\epsilon/2} \cap R_V(I)^{(1,1+\epsilon^2/8)}$.

Now suppose $((x',t'),s') \in U'$. We will show that the integral curve starting at $(y',t') := \Phi((x',t'),s')$ never leaves $Y_{\epsilon}$.

First note that because $|t'_i - t_i| < \epsilon/2 < \hat{t}/2$, we have $\hat{t'} > \hat{t}/2$. We also have $r(y',t') < 1+ \epsilon^2/8$. According to the proof of Lemma~\ref{lem:curve}, along the integral curve starting at $(y',t')$ the function $r$ decreases at rate at least $\hat{t'}$. That curve therefore must leave $R_V(I)^{>1}$ after `time' at most
\[ \frac{\epsilon^2/8}{\hat{t'}} < \frac{\epsilon^2/8}{\hat{t}/2}< \frac{\epsilon}{4}. \]
Moreover, the derivative of that curve is bounded in each coordinate by $r(y',t') < 2$, and so the entire curve can travel a distance at most $\epsilon/2$ in each coordinate. Since $(y',t') \in Y_{\epsilon/2}$, it follows that the entire integral curve is contained in $Y_{\epsilon}$.

Now let $U$ be the subset of $R_V(I) \times [0, \infty)$ given by
\[ U = \{((x',t'),s'') \; | \; ((x',t'),s') \in U' \text{ for some $s' < s''$} \}. \]
Then the above argument implies that $\Phi(U) \subseteq Y_{\epsilon}$. Since we have $((x,t),s_0) \in U'$ and $s_0 < s$, $U$ is a neighbourhood of $((x,t),s)$ as required.

\underline{Case 2}: $r(x,t) = 1$.

In this case, we have $x = y$. We choose
\[ U = (Y_{\epsilon/2} \cap R_V(I)^{(0,1+\epsilon^2/8)}) \times [0,\infty). \]
The argument in Case 1 shows that every integral curve starting at a point in $(Y_{\epsilon/2} \cap R_V(I)^{(1,1+\epsilon^2/8)})$ remains within $Y_{\epsilon}$. Therefore $\Phi(U) \subseteq Y_{\epsilon}$. 

This completes the proof that $\Phi$ is continuous.

Next suppose that $(x,t) \in RD_V(I)$, i.e.\ that $|x_i - x_j| \geq t_i+t_j$ for all $i,j \in I$. Suppose that $(y,t) = \Phi((x,t),s) \notin RD_V(I)$ for some $s$, so that $(y,t)$ is on the integral curve for $X$ starting at $(x,t)$, and there are $i,j \in I$ such that
\[ |y_i - y_j| < t_i+t_j. \]
Without loss of generality, we can assume $(x,t)$ is the last point at which this integral curve leaves the closed subspace of $R_V(I)$ determined by the condition $|x_i - x_j| \geq t_i+t_j$ before reaching $(y,t)$. It then follows that for all points $(x',t) = \Phi((x,t),s')$ with $0 < s' \leq s$, we have
\[ |x'_i - x'_j| < t_i+t_j. \]
It follows also that $i,j$ are in the same piece of the partition $\pi(x',t)$ for all $s' \in [0,s]$.

Now consider the vector field $X$ which defines the flow $\Phi$. We have
\[ X(x',t) = \sum_{\pi} \rho_{\pi} X^{\pi}(x',t). \]
Since $\rho_{\pi}$ is supported on $U_{\pi}$, each non-zero term in this sum has $(x',t) \in U_{\pi}$, and so $\pi(x',t)$ is a refinement (or equal to) $\pi$. It follows that $i,j$ are in the same piece of the partition $\pi$, and so $[i]_{\pi} = [j]_{\pi}$. Therefore $X^{\pi}_i = X^{\pi}_j$ for each non-zero term in the sum above, and so
\[ X_i(x',t) = X_j(x',t). \]
Then $x'_i - x'_j$ is constant along the integral curve which defines $\Phi((x,t),s')$ for $s' \in [0,s]$, but this fact contradicts the changing value of $|x'_i - x'_j|$. Therefore $\Phi((x,t),s) \in RD_V(I)$.
\end{proof}

\begin{proposition} \label{prop:defret}
The function
\[ H: RD_V(I) \times [0,1) \to RD_V(I) \]
given by
\[ H((x,t),u) = \Phi((x,t),-\log(1-u)) \]
extends to a (fibrewise over $\Delta(I)$) deformation retraction of $RD_V(I)$ onto $D_V(I)$.
\end{proposition}
\begin{proof}
It follows from Lemma~\ref{lem:curve} that for each $(x,t) \in RD_V(I)$, there is some $u_0 < 1$ such that $H((x,t),u_0) \in D_V(I)$, and that $H((x,t),u) = H((x,t),u_0)$ for all $u_0 \leq u < 1$. We therefore extend $H$ by setting
\[ H((x,t),1) := H((x,t),u_0) = \lim_{u \to 1} H((x,t),u) \in D_V(I). \]
Since $H$ is continuous, this extended function is too. Finally, we note that $H((x,t),0) = \Phi((x,t),0) = (x,t)$ and that if $(x,t) \in D_V(I)$, then $H((x,t),u) = (x,t)$ for all $u \in [0,1]$. Thus the extended function $H$ is a deformation retraction as required.
\end{proof}

\section{The Fulton-MacPherson operad} \label{sec:FM}

The Fulton-MacPherson operad $F_V$ is another model for the topological operad $E_V$, consisting of certain compactifications of the configuration spaces of points in $V$, modulo translation and positive scaling. These compactifications were defined by Fulton and MacPherson \cite{fulton/macpherson:1994} based on ideas of Axelrod and Singer~\cite{axelrod/singer:1994}, and were given an operad structure by Getzler and Jones~\cite{getzler/jones:1994}. We will use a different description of these spaces, inspired by work of Barber~\cite{barber:2017}, which fits better with our later constructions.

\begin{definition} \label{def:FM}
Let $V$ be a finite-dimensional real vector space, and let $I$ be a finite set. We define the topological space $F_V(I)$ as follows: a point $y \in F_V(I)$ assigns to each subset $J \subseteq I$ with $|J| \geq 2$ a $J$-tuple $y(J)$ of vectors in $V$, not all equal, defined modulo translation and positive scaling, such that
\begin{itemize}
  \item for $J \subseteq J'$, we have either that $y(J')|_J \equiv y(J)$, modulo translation and positive scaling, or that $y(J')|_J$ is a constant $J$-tuple.
\end{itemize}
Here $y(J')|_J$ denotes the restricted $J$-tuple $(y(J')_j)_{j \in J}$. The topology on $F_V(I)$ is the subspace topology relative to $\prod_{J \subseteq I} V^J/(V \times (0,\infty))$.
\end{definition}

\begin{remark}
Suppose $y \in F_V(I)$ is such that all the vectors in $y(I)$ are distinct. Then $y(J)$ is completely determined for all $J \subseteq I$ by the restriction condition. The subspace of $F_V(I)$ consisting of such points is therefore homeomorphic to the ordinary configuration space of $I$-tuples in $V$, modulo translation and positive scaling. We denote this subspace of $F_V(I)$ by $\mathring{F}_V(I)$.
\end{remark}

\begin{example} \label{ex:FV}
Here is an example of a specific point $y$ in the space $F_{\R^2}(\un{4})$, where $\un{4} = \{1,2,3,4\}$. First, we have $y(\un{4}) = (y_1,y_2,y_3,y_4)$: a $4$-tuple of points in $\R^2$, for example:

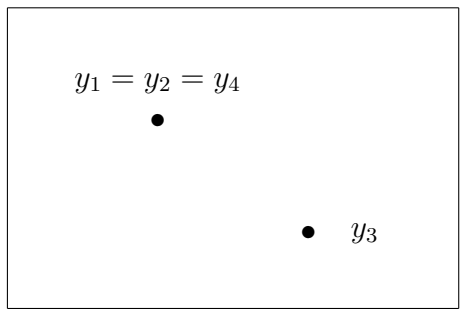
\begin{center}
\begin{tikzpicture}
	\begin{pgfonlayer}{nodelayer}
		\node [style=none] (0) at (0, 0.5) {$\bullet$};
		\node [style=none] (1) at (2, -1) {$\bullet$};
		\node [style=none] (2) at (0, 1) {};
		\node [style=none] (5) at (0, 1) {$y_1 = y_2 = y_4$};
		\node [style=none] (6) at (-2, 2) {};
		\node [style=none] (7) at (-2, -2) {};
		\node [style=none] (8) at (4, -2) {};
		\node [style=none] (9) at (4, 2) {};
		\node [style=none] (10) at (2.75, -1) {$y_3$};
	\end{pgfonlayer}
	\begin{pgfonlayer}{edgelayer}
		\draw (6.center) to (7.center);
		\draw (9.center) to (6.center);
		\draw (8.center) to (7.center);
		\draw (9.center) to (8.center);
	\end{pgfonlayer}
\end{tikzpicture}
\end{center}

For any subset $J \subseteq \un{4}$ that includes $3$, the $J$-tuple $y(J)$ is determined by $y(\un{4})$; the picture is the same with points removed. However, the point $y$ also includes an object $y(\{1,2,4\})$ which could be any configuration of three points in $\R^2$ (not all equal). The intuition here is that we have zoomed infinitely far into the previous picture so that we can now distinguish the points $y_1$, $y_2$ and $y_4$. For example, we might have:

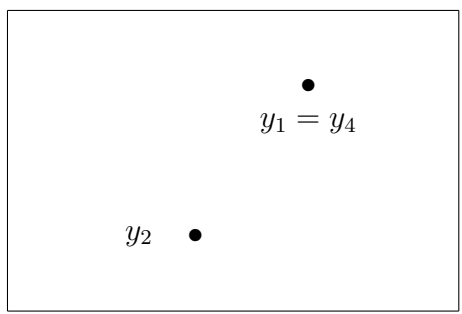
\begin{center}
\begin{tikzpicture}
	\begin{pgfonlayer}{nodelayer}
		\node [style=none] (0) at (3, 0) {$\bullet$};
		\node [style=none] (1) at (1.5, -2) {$\bullet$};
		\node [style=none] (2) at (0.75, -2) {$y_2$};
		\node [style=none] (3) at (3, -0.5) {$y_1 = y_4$};
		\node [style=none] (6) at (-1, 1) {};
		\node [style=none] (7) at (-1, -3) {};
		\node [style=none] (8) at (5, -3) {};
		\node [style=none] (9) at (5, 1) {};
	\end{pgfonlayer}
	\begin{pgfonlayer}{edgelayer}
		\draw (6.center) to (7.center);
		\draw (9.center) to (6.center);
		\draw (8.center) to (7.center);
		\draw (9.center) to (8.center);
	\end{pgfonlayer}
\end{tikzpicture}
\end{center}

The only additional information carried by the point $y$ that is still undetermined is the relationship between $y_1$ and $y_4$, i.e.\ the $2$-tuple $y(\{1,4\})$. For example:

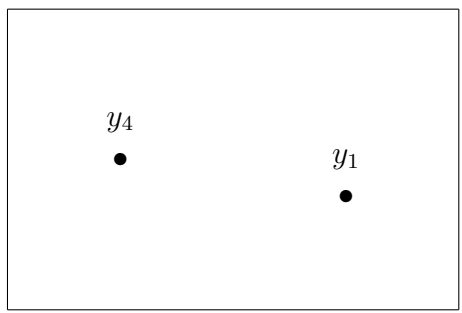
\begin{center}
\begin{tikzpicture}
	\begin{pgfonlayer}{nodelayer}
		\node [style=none] (0) at (3.5, -1.5) {$\bullet$};
		\node [style=none] (1) at (0.5, -1) {$\bullet$};
		\node [style=none] (2) at (0.5, -0.5) {$y_4$};
		\node [style=none] (3) at (3.5, -1) {$y_1$};
		\node [style=none] (6) at (-1, 1) {};
		\node [style=none] (7) at (-1, -3) {};
		\node [style=none] (8) at (5, -3) {};
		\node [style=none] (9) at (5, 1) {};
	\end{pgfonlayer}
	\begin{pgfonlayer}{edgelayer}
		\draw (6.center) to (7.center);
		\draw (9.center) to (6.center);
		\draw (8.center) to (7.center);
		\draw (9.center) to (8.center);
	\end{pgfonlayer}
\end{tikzpicture}
\end{center}
\end{example}

The $J$-tuple $y(J)$ for all other subsets $J \subseteq \{1,2,3,4\}$ is now determined by $y(\un{4})$, $y(\{1,2,4\})$ and $y(\{1,4\})$ together with the restriction condition in Definition~\ref{def:FM}.

\begin{remark}
A common approach to the description of the Fulton-MacPherson space $F_V(I)$ is to only specify the non-redundant information; for example, in the previous example, that would be by giving only the three illustrated configurations: a point in $\mathring{F}_V(2) \times \mathring{F}_V(2) \times \mathring{F}_V(2)$. It is more convenient for us to have all the relationships between the points $y_i$ specified for all possible subsets of $I$, so we build that information into our definition.
\end{remark}

\begin{remark}
An explicit construction of the Fulton-MacPherson operad was given by Sinha in \cite{sinha:2004} where the space we are calling $F_{\R^m}(n)$ was labelled $\tilde{C}_{n}[\R^m]$ and was defined to be a certain closed subspace of $(S^{m-1})^{C_2(n)} \times [0,\infty]^{C_3(n)}$ where $C_r(n)$ denotes the set of $r$-element subsets of the finite set $n$. To match up our definition with Sinha's, we define a map
\[ F_{\R^m}(n) \to (S^{m-1})^{C_2(n)} \times [0,\infty]^{C_3(n)} \]
by sending the point $y$ to the collection consisting of:
\begin{itemize}
  \item for each $2$-element subset $J = \{j_1,j_2\} \subseteq n$, the point
  \[ \frac{y(J)_{j_1} - y(J)_{j_2}}{|y(J)_{j_1} - y(J)_{j_2}|} \in S^{m-1}; \]
  \item for each $3$-element subset $K = \{k_1,k_2,k_3\} \subseteq n$, the point
  \[ \frac{|y(K)_{k_1} - y(K)_{k_2}|}{|y(K)_{k_1} - y(K)_{k_3}|} \in [0,\infty]. \]
\end{itemize}
This map determines a homeomorphism of $F_{\R^m}(n)$ with Sinha's $\tilde{C}_{n}[\R^m]$. \end{remark}

\begin{remark}
The Fulton-MacPherson space $F_V(I)$ is stratified by the poset of $I$-labelled trees to be introduced in section~\ref{sec:map}. The tree corresponding to a particular point $y \in F_V(I)$ consists of those subsets $J$ for which the term $y(J)$ is not determined by any larger subset. In Example~\ref{ex:FV}, the relevant tree would have non-leaf edges
\[ \{\{1,2,3,4\}, \{1,2,4\}, \{1,4\} \}. \]
Pictorially, this is the following tree:

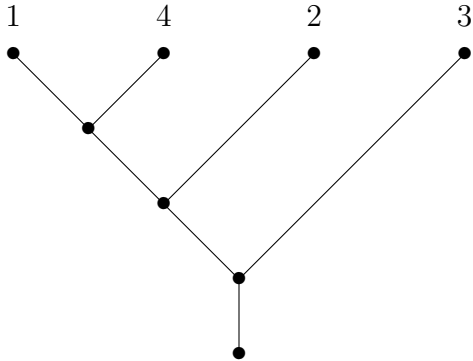
\begin{center}
\begin{tikzpicture}
	\begin{pgfonlayer}{nodelayer}
		\node [style=none] (0) at (0, 0) {$\bullet$};
		\node [style=none] (1) at (0, 1) {$\bullet$};
		\node [style=none] (2) at (-1, 2) {$\bullet$};
		\node [style=none] (3) at (3, 4) {$\bullet$};
		\node [style=none] (4) at (-2, 3) {$\bullet$};
		\node [style=none] (5) at (1, 4) {$\bullet$};
		\node [style=none] (6) at (-3, 4) {$\bullet$};
		\node [style=none] (7) at (-1, 4) {$\bullet$};
		\node [style=none] (8) at (-3, 4.5) {$1$};
		\node [style=none] (9) at (-1, 4.5) {$4$};
		\node [style=none] (10) at (1, 4.5) {$2$};
		\node [style=none] (11) at (3, 4.5) {$3$};
	\end{pgfonlayer}
	\begin{pgfonlayer}{edgelayer}
		\draw (6.center) to (1.center);
		\draw (7.center) to (4.center);
		\draw (5.center) to (2.center);
		\draw (3.center) to (1.center);
		\draw (1.center) to (0.center);
	\end{pgfonlayer}
\end{tikzpicture}
\end{center}
\end{remark}

\begin{definition} \label{def:FM-operad}
We now put an operad structure on $F_V$. The intuition behind the operad composition maps is that we are inserting one configuration infinitesimally in place of one point of another. This construction is easy to define using our version of $F_V(I)$.

Given $y \in F_V(I)$, $z \in F_V(J)$ and $i \in I$, we define $y \circ_i z \in F_V(I \cup_i J)$ by setting, for $K \subseteq I \cup_i J$:
\begin{equation} \label{eq:FM-operad} (y \circ_i z)(K) := \begin{cases} z(K) & \text{if $K \subseteq J$}; \\ \pi^*y(\pi(K)) & \text{if $K \nsubseteq J$}; \end{cases} \end{equation}
where $\pi: I \cup_i J \to I$ is the map that sends each element of $J$ to $i$, and $\pi^*$ denotes the pullback of a $\pi(K)$-tuple along $\pi$.
\end{definition}

\begin{theorem} \label{thm:FM}
The construction described in Definition~\ref{def:FM-operad} make $F_V$ into an operad in the category of topological spaces with an action of the general linear group $GL(V)$.
\end{theorem}
\begin{proof}
We check first that Definition~\ref{def:FM-operad} provides a well-defined point $y \circ_i z \in F_V(I \cup_i J)$. Notice that if $K \nsubseteq J$, then $|\pi(K)| \geq 2$, and not all points in $y(\pi(K))$ are equal. Since $\pi$ is surjective, it follows that not all points in $\pi^*y(\pi(K))$ are equal. Suppose that $K \subseteq K'$; we are required to show that
\[ (y \circ_i z)(K) \equiv (y \circ_i z)(K')|_K \]
or else the latter object is a constant $K$-tuple.

If $K' \subseteq J$, then this is because $z(K) \equiv z(K')|_K$. If $K \nsubseteq J$, then it is because $\pi(K) \subseteq \pi(K')$ and
\[ \pi^*y(\pi(K)) \equiv \pi^*(y(\pi(K'))|_{\pi(K)}) = \pi^*y(\pi(K'))|_K. \]
Finally, if $K \subseteq J$ but $K' \nsubseteq J$, then
\[ (y \circ_i z)(K')|_K = \pi^*y(\pi(K'))|_K \]
is a constant $K$-tuple.

We thus conclude that (\ref{eq:FM-operad}) defines a function
\[ \circ_i: F_V(I) \times F_V(J) \to F_V(I \cup_i J) \]
and it is straightforward to check that $\circ_i$ is continuous since this condition can be checked on each subset $K \subseteq I \cup_i J$ separately. We leave the reader to check the associativity conditions for these composition maps to determine an operad structure.
\end{proof}

The second author proved the following result in~\cite[4.9]{salvatore:2001}, thus showing that the Fulton-MacPherson operad is indeed a model for the topological $E_n$-operad.

\begin{proposition}
The operad $F_V$ is cofibrant (in the projective model structure on reduced operads of topological spaces) and is $O(V)$-equivariantly equivalent to $E_V$.
\end{proposition}
\begin{proof}
The $O(V)$-equivariance is not explicitly mentioned in \cite{salvatore:2001}, but the equivalence of operads $WE_V \weq F_V$ constructed there to prove this result has the necessary equivariance anyway.
\end{proof}

\begin{remark} \label{rem:FM}
In working with the operad $F_V$, it will be convenient to choose specific representatives $y(J)$ of the $J$-tuples that make up a point $y \in F_V(I)$. However, these choices will be fibred over the simplex (co)operad, that is, they will depend on a fixed set of `weights' $t \in (0,\infty)^I$ with $\sum_i t_i = 1$. So, unless stated otherwise, and assuming a point $t \in \Delta(I)$ is given, we will assume that, for each subset $J \subseteq I$, the $J$-tuple $y(J) \in V^J$ satisfies the following conditions:
\begin{itemize}
  \item a \emph{weighted barycentre condition}:
  \[ \sum_{j \in J} t_j y(J)_j = 0, \]
  which fixes $y(J)$ up to positive scaling;
  \item a \emph{weighted norm condition}
  \[ \frac{\sum_{j \in J} t_j \; |y(J)_j|}{\sum_{j \in J} t_j} = 1 \]
  which fixes $y(J) \in V^J$.
\end{itemize}
\end{remark}

In general the definition of $y \circ_i z$ in (\ref{eq:FM-operad}) does not satisfy the weighted barycentre and norm conditions, though it does for certain subsets $K \subseteq I \cup_i J$.

\begin{lemma} \label{lem:FM-norm}
Suppose that $y$ satisfies the weighted barycentre and norm conditions with respect to $t \in \Delta(I)$, and that $z$ satisfies those conditions with respect to $u \in \Delta(J)$. Let $K$ be a subset of $I \cup_i J$ satisfying one of the following three conditions: (1) $K \subseteq J$; (2) $K \supseteq J$; (3) $J \cap K = \varnothing$. Then $(y \circ_i z)(K)$ satisfies the weighted barycentre and norm conditions with respect to the weighting $t \cdot_i u \in \Delta(I \cup_i J)$.
\end{lemma}
\begin{proof}
We deal with the three cases separately: (1) for $K \subseteq J$, we have
\[ \sum_{k \in K} (t \cdot_i u)_k(y \circ_i z)(K)_k = \sum_{k \in K} t_iu_kz(K)_k = t_i \cdot 0 = 0 \]
and
\[ \sum_{k \in K} (t \cdot_i u)_k \; |(y \circ_i z)(K)_k| = \sum_{k \in K} t_iu_k \; |z(K)_k| = t_i \cdot \sum_{k \in K} u_k = \sum_{k \in K} (t \cdot_i u)_k; \]
(2) for $K \supseteq J$, we have
\[ \begin{split} \sum_{k \in K} (t \cdot_i u)_k(y \circ_i z)(K)_k &= \sum_{k \in K \setminus J} t_k y(\pi(K))_k + \sum_{k \in J} t_i u_k y(\pi(K))_i \\ &= \sum_{k \in K \setminus J} t_k y(\pi(K))_k + t_i y(\pi(K))_i = \sum_{k \in \pi(K)} t_k y(\pi(K))_k = 0 \end{split} \]
and
\[ \begin{split} \sum_{k \in K} (t \cdot_i u)_k \; |(y \circ_i z)(K)_k| &= \sum_{k \in K \setminus J} t_k \; |y(\pi(K))_k| + \sum_{k \in J} t_i u_k \; |y(\pi(K))_i| \\ &= \sum_{k \in K \setminus J} t_k \; |y(\pi(K))_k| + t_i |y(\pi(K))_i| \\ &= \sum_{k \in \pi(K)} t_k \; |y(\pi(K))_k| = \sum_{k \in \pi(K)} t_k = \sum_{k \in K \setminus J} t_k + t_i = \sum_{k \in K} (t \cdot_i u)_k; \end{split} \]
and finally, for (3) $J \cap K = \varnothing$, these same last equations apply with the terms involving $k \in J$, and $t_i$, removed.
\end{proof}

\section{Operadic suspension} \label{sec:susp}

Recall that for an operad $\mathsf{P}$ of chain complexes there is a simple suspension operation $s$ for which $s\mathsf{P}(n)$ is given by shifting the chain complex $\mathsf{P}(n)$ up by degree $n-1$, introducing signs as necessary. An $s\mathsf{P}$-algebra can then be identified with a $\mathsf{P}$-algebra shifted down in degree.

In~\cite{arone/kankaanrinta:2014}, Arone and Kankaanrinta described a topological version of the operadic suspension based on what they called a `sphere operad', that is an operad $S$ of pointed spaces for which the $n$th term is homeomorphic to the sphere $S^{n-1}$, with composition maps that are homeomorphisms. In~\cite[1.1]{arone/kankaanrinta:2014}, they laid out some desirable properties for such an operad, and then proved that such an operad exists. They defined the `suspension' for an operad $P$ of pointed spaces (or spectra) to be given by taking a termwise smash product with $S$. As described in the Introduction to \cite{arone/kankaanrinta:2014}, the key property of $S$ as regards the operadic suspension is that $\Sigma^\infty S$ is equivalent, as an operad of spectra, to the coendomorphism operad of the spectrum $\Sigma^\infty S^1$.

In this paper, we use the $V$-sphere (co)operad $S_V$ of Definition~\ref{def:SV} in place of $S$. The operad $S_V$ does not enjoy all the properties described in \cite[1.1]{arone/kankaanrinta:2014}: for example, its terms are only homotopy equivalent to spheres, not homeomorphic. Nonetheless, the following result justifies our use of $S_V$ to construct a $V$-indexed suspension for (co)operads of spectra. Moreover, the advantage that $S_V$ has over, say, the smash product of some number of copies of $S$ is that $S_V$ retains an action of the general linear group $GL(V)$.

\begin{proposition}
Let $V$ be a finite-dimensional real vector space with one-point compactification $S^V$. Then the operad of spectra $\Sigma^\infty S_V$ is equivalent to the coendomorphism operad of the spectrum $\Sigma^\infty S^V$.
\end{proposition}
\begin{proof}
We construct maps of pointed spaces
\[ S^V \smsh S_V(I) \to (S^V)^{\smsh I} \]
given (away from the basepoint) by
\[ (v,(x,t)) \mapsto \left(\frac{v+x_i}{t_i}\right)_{i \in I} \]
for $v \in V$, $(x,t) \in R_V(I)$. These induce maps
\[ \Sigma^\infty S_V(I) \to \Map(\Sigma^\infty S^V, (\Sigma^\infty S^V)^{\smsh I}) \]
which form the desired equivalence of operads. The required associativity conditions follow from the commutativity of diagrams such as the following:
\[ \begin{diagram}
  \node{S^V \smsh S_V(I) \smsh S_V(J)} \arrow{s} \arrow{e} \node{S^V \smsh S_V(I \cup_i J)} \arrow{s} \\
  \node{(S^V)^{\smsh I} \smsh S_V(J)} \arrow{e} \node{(S^V)^{\smsh I \cup_i J}}
\end{diagram} \]
given by
\[ \begin{diagram}
  \node{(v, (x,t), (y,u))} \arrow{s} \arrow{e} \node{(v, (x+_i ty,t \cdot_i u))} \arrow{s} \\
  \node{\left(\left(\frac{v+x_{i'}}{t_{i'}}\right)_{i' \in I}, (y,u)\right)} \arrow{e} \node{\left(\left(\frac{v+x_{i'}}{t_{i'}}\right)_{i' \in I - \{i\}}, \left(\frac{v+x_i+t_iy_j}{t_iu_j}\right)_{j \in J}\right).}
\end{diagram} \qedhere \]
\end{proof}

We now define our version of operadic suspension using $S_V$.

\begin{definition}
Let $\mathbf{P}$ be an operad or cooperad of spectra, and let $V$ be a finite-dimensional real vector space. The \emph{$V$-suspension} of $\mathbf{P}$ is the operad or cooperad $\Sigma^V\mathbf{P}$ given by
\[ (\Sigma^V\mathbf{P})(I) := S_V(I) \smsh \mathbf{P}(I) \]
for each finite set $I$ with $|I| \geq 2$. The necessary structure maps are given by combining those of $\mathbf{P}$ with the relevant homeomorphisms from Proposition~\ref{prop:S}. Note that if $\mathbf{P}$ has a $GL(V)$-action, then $\Sigma^V\mathbf{P}$ can be given the diagonal $GL(V)$-action.
\end{definition}

We can also define an operadic desuspension, but this construction is restricted to operads.

\begin{definition}
Let $\mathbf{P}$ be an operad of spectra, and let $V$ be a finite-dimensional real vector space. The \emph{$V$-desuspension} of $P$ is the operad of spectra $\Sigma^{-V}\mathbf{P}$ given by
\[ (\Sigma^{-V}\mathbf{P})(I) := \Map(S_V(I),\mathbf{P}(I)) \]
for each finite set $I$ with  $|I| \geq 2$. The necessary structure maps are then induced by the cooperad structure maps for $S_V$ and the operad structure maps for $\mathbf{P}$. If $\mathbf{P}$ is a $GL(V)$-operad, then so is $\Sigma^{-V}\mathbf{P}$.
\end{definition}

While we have defined operadic (de)suspension using the sphere (co)operad $S_V$, our duality map will involve a certain quotient of $S_V$ that is still homotopy equivalent to it.

\begin{definition} \label{def:S}
Let $V$ be a finite-dimensional real normed vector space. For a nonempty finite set $I$, we define
\[ \mathring{S}_V(I) := \{ (x,t) \in R_V(I) \; | \; |x_i| < t_i, \; |x_i - x_j| < \min\{t_i,t_j\} \text{ for all $i,j \in I$} \}. \]
We can visualise a point in $\mathring{S}_V(I)$ as an $I$-indexed collection of discs in $V$ such that the interior of each disc contains the origin as well as the center of every other disc.
\end{definition}

\begin{proposition} \label{prop:S-coop}
The subsets $\mathring{S}_V(I)$ are open and contractible, and form an $O(V)$-sub-cooperad of $R_V$ for which the decomposition maps are open embeddings.
\end{proposition}
\begin{proof}
Each $\mathring{S}_V(I)$ is a (fibrewise over $\Delta(I)$) star-shaped open subset of $R_V(I)$, hence contractible. The cooperad structure maps for $R_V$ are homeomorphisms, so it is sufficient to show that those structure maps restrict to $\mathring{S}_V$.

So take $(x,t) \in \mathring{S}_V(I \cup_i J)$. Consider first $(x/J,t/J)$. We have
\[ |(x/J)_{i'}| = |x_{i'}| < t_{i'} = (t/J)_{i'}, \quad \text{for $i' \neq i$}, \]
and
\[ |(x/J)_i| = \left|\frac{\sum_{j \in J} t_jx_j}{\sum_{j \in J} t_j} \right| < \frac{\sum_{j \in J} t_j^2}{\sum_{j \in J} t_j} \leq \sum_{j \in J} t_j = (t/J)_i. \]
For $i',i'' \neq i$, we have
\[ |(x/J)_{i'} - (x/J)_{i''}| = |x_{i'} - x_{i''}| < \min\{t_{i'},t_{i''}\} = \min\{(t/J)_{i'},(t/J)_{i''}\} \]
and
\[ |(x/J)_{i'} - (x/J)_i| = \left|x_{i'} - \frac{\sum_{j \in J} t_jx_j}{\sum_{j \in J}t_j}\right| \leq \frac{\sum_{j \in J} t_j |x_{i'} - x_j|}{\sum_{j \in J}t_j}. \]
This is less than both
\[ \frac{\sum_{j \in J}t_jt_{i''}}{\sum_{j \in J}t_j} = t_{i''} = (t/J)_{i''} \]
and
\[ \frac{\sum_{j \in J}t_j^2}{\sum_{j \in J}t_j} \leq \sum_{j \in J} t_j = (t/J)_i \]
as required.

Now consider $(x|J,t|J)$. We have (using the previous calculation)
\[ |(x|J)_j| = \frac{|x_j - x_J|}{\sum_{j' \in J}t_{j'}} < \frac{t_j}{\sum_{j' \in J}t_{j'}} = (t|J)_j \]
and
\[ |(x|J)_j - (x|J)_{j'}| = \frac{|x_j - x_{j'}|}{\sum_{j'' \in J}t_{j''}} < \frac{\min\{t_j,t_{j'}\}}{\sum_{j'' \in J}t_{j''}} = \min\{(t|J)_j,(t|J)_{j'}\}. \qedhere \]
\end{proof}

\begin{definition} \label{def:barS}
For a finite-dimensional real normed vector space $V$, we define the quotient spaces
\[ \bar{S}_V(I) := S_V(I)/(S_V(I) - \mathring{S}_V(I)). \]
\end{definition}

\begin{proposition} \label{prop:barS}
The operad composition maps for $S_V$ pass to the quotients and make $\bar{S}_V$ into an operad of pointed spaces. The quotient maps $S_V(I) \to \bar{S}_V(I)$ form an equivalence of operads
\[ S_V \weq \bar{S}_V. \]
\end{proposition}
\begin{proof}
The first claim follows from Proposition~\ref{prop:S-coop}. For the rest, we must show that the quotient map $q: S_V(I) \to \bar{S}_V(I)$ is a weak homotopy equivalence. To see this claim, we show that $q$ is homotopic to a homeomorphism.

Given $(x,t) \in R_V(I)$ we have a continuous function given by
\[ b(x,t) := \max_{i,j \in I} \left\{ \frac{|x_i|}{t_i}, \frac{|x_i - x_j|}{\min\{t_i,t_j\}} \right\} \in [0,\infty). \]
We define a homotopy $H: S_V(I) \times [0,1] \to \bar{S}_V(I)$ by
\[ H((x,t),s) = \left( \left( 1 - \frac{sb(x,t)}{1+b(x,t)} \right) x, t \right). \]
For $s = 0$, the homotopy $H$ restricts to $q$, and for $s = 1$ it restricts to the homeomorphism whose inverse is the function $\bar{S}_V(I) \to S_V(I)$ given by
\[ (y,t) \mapsto \left( \frac{1}{1-b(y,t)}y, t \right). \qedhere \]
\end{proof}

It follows from Proposition~\ref{prop:barS} that $\bar{S}_V$ is a suitable operad to use in place of $S_V$ for the operadic suspension $\Sigma^V\mathbf{P}$ of an operad of spectra (or pointed spaces).

We conclude this section with an observation about the barycentres of configurations in $\mathring{S}_V(I)$ that will be useful later.

\begin{lemma} \label{lem:barS}
Suppose $K, K' \subseteq I$ with $K \cap K' = \varnothing$. Then, for any $(x,t) \in \mathring{S}_V(I)$, we have
\[ |x_K| < t_K, \quad |x_K - x_{K'}| < \min\{t_K,t_{K'}\}. \]
\end{lemma}
\begin{proof}
We have
\[ |x_K| = \left| \frac{\sum_{k \in K} t_kx_k}{\sum_{k \in K} t_k} \right| \leq \frac{\sum_{k \in K} t_k|x_k|}{\sum_{k \in K} t_k} < \frac{\sum_{k \in K} t_k^2}{\sum_{k \in K} t_k}  < \sum_{k \in K} t_k = t_K \]
and
\[ |x_K - x_{K'}| = \left| \frac{\sum_{k \in K} t_k x_k}{\sum_{k \in K} t_k} - \frac{\sum_{k' \in K} t_{k'}x_{k'}}{\sum_{k' \in K} t_{k'}} \right| \leq \frac{\sum_{k,k'} t_k t_{k'}|x_k - x_{k'}|}{\sum_{k,k'} t_k t_{k'}} < \frac{\sum_{k'} t_{k'} \sum_{k} t_k^2}{\sum_{k'} t_{k'} \sum_k t_k} < \sum_k t_k   \]
and similarly $|x_K - x_{K'}| < t_{K'}=t_K$.
\end{proof}

\section{Bar-cobar duality for operads} \label{sec:bar-cobar}

\begin{definition} \label{def:bar}
Let $\mathbf{P}$ be an operad of spectra in the sense of Definition~\ref{def:operad-spectra}. (Recall that, in this paper, all operads are automatically reduced.) The \emph{bar construction} on $\mathbf{P}$ is a cooperad of spectra, denoted $B\mathbf{P}$ that can be described (as a symmetric sequence) in a number of ways:
\begin{itemize}
  \item $B\mathbf{P}$ is the geometric realization of the (reduced) simplicial bar construction on the operad $\mathbf{P}$ considered as a monoid with respect to the composition product $\circ$ of symmetric sequences: that is
      \[ B\mathbf{P} \isom B(\mathbf{1},\mathbf{P},\mathbf{1}) := | \mathbf{1} \Leftarrow \mathbf{P} \Lleftarrow \mathbf{P} \circ \mathbf{P} \dots | \]
      where $\mathbf{1}$ denotes the trivial operad of spectra (which is given by $\mathbf{1}(I) = *$ for every finite set $I$ with $|I| \geq 2$) with a $(\mathbf{P},\mathbf{P})$-bimodule structure induced by the operad augmentation map $\mathbf{P} \to \mathbf{1}$;
  \item $B\mathbf{P}$ is the derived composition product
      \[ B\mathbf{P} \homeq \mathbf{1} \circ^{\mathbb{L}}_{\mathbf{P}} \mathbf{1} \]
      of the trivial operad with itself over $\mathbf{P}$;
  \item $B\mathbf{P}$ is a model for the termwise-suspension of the `derived indecomposables' of $\mathbf{P}$: that is
      \[ B\mathbf{P} \isom \Sigma ( W\mathbf{P}/\partial W\mathbf{P} )\]
      where $W\mathbf{P}$ denotes the Boardman-Vogt $W$-construction on $\mathbf{P}$ (a cofibrant replacement for the operad $\mathbf{P}$) and $\partial W\mathbf{P}$ denotes the sub-object of `decomposables' inside $W\mathbf{P}$, that is the combined image of all the composition maps
      \[ W\mathbf{P}(I) \times W\mathbf{P}(J) \to W\mathbf{P}(I \cup_i J); \]
  \item for each finite set $I$ with $|I| \geq 2$, we can describe $B\mathbf{P}(I)$ as a coend over a certain poset $\mathsf{Tree}_I$ of $I$-labelled trees (to be described in more detail in section~\ref{sec:map}) of the form
      \[ B\mathbf{P}(I) \isom \bar{w}(T) \smsh_{T \in \mathsf{Tree}_I} \mathbf{P}(T) \]
      where $\bar{w}(T)$ is the space of ways to assign non-negative real numbers to the non-leaf edges of $T$, with the limiting case where any edge is assigned $\infty$, or where the root is assigned $0$, identified to a single basepoint.
\end{itemize}
The authors independently constructed a cooperad structure on the symmetric sequence $B\mathbf{P}$, in \cite{salvatore:1999} and \cite{ching:2005} respectively. We recall the details of this structure in Definition~\ref{def:bar1} below.

Note also that all of the above descriptions, including the cooperad structure, apply equally well to the bar construction on an operad of pointed spaces (or of unpointed spaces with disjoint basepoints added).
\end{definition}

\begin{proposition} \label{prop:bar-cobar}
The bar construction of Definition~\ref{def:bar} determines a functor
\[ B: \mathsf{Op}(\spectra) \to \mathsf{Coop}(\spectra) \]
from the category of (reduced) operads of spectra to the category of cooperads. An entirely dual procedure determines the \emph{cobar construction}, a functor
\[ C: \mathsf{Coop}(\spectra) \to \mathsf{Op}(\spectra). \]
\end{proposition}

The duality between the bar and cobar constructions was described in \cite{ching:2012} where the following result was proved.

\begin{theorem} \label{thm:bar-cobar}
There is a Quillen equivalence
\[ \mathbb{B} : \mathsf{Op}(\spectra) \rightleftarrows \mathsf{QCoop}(\spectra) : \mathbb{C} \]
between Quillen model categories of (reduced) operads and \emph{quasi-}cooperads (a generalization of the notion of cooperad to be described in section~\ref{sec:quasi}) of spectra where:
\begin{itemize}
  \item in $\mathsf{Op}(\spectra)$, weak equivalences and fibrations are detected levelwise on the underlying spectra;
  \item each object in $\mathsf{QCoop}(\spectra)$ is weakly equivalent to a cooperad;
  \item for a cofibrant operad $\mathbf{P}$, the quasi-cooperad $\mathbb{B}\mathbf{P}$ is equivalent to the cooperad $B\mathbf{P}$ given by the bar construction on $\mathbf{P}$;
  \item the right adjoint $\mathbb{C}$ is an extension to quasi-cooperads of the cobar construction $C$.
\end{itemize}
\end{theorem}

The bar-cobar duality can also be described purely in terms of operads using the following definition.

\begin{definition} \label{def:kosuzl}
Let $\mathbf{P}$ be an operad of spectra. We say that $\mathbf{P}$ is \emph{termwise-finite} if each spectrum $\mathbf{P}(I)$ is equivalent to a finite CW-spectrum. In this case, we define the \emph{(derived) Koszul dual} of $\mathbf{P}$ to be the operad of spectra $K\mathbf{P}$ given by
\[ (K\mathbf{P})(I) := \Map(B\mathbf{P}'(I),\mathbf{S}) \]
where $\mathbf{P}'$ is a cofibrant replacement of $\mathbf{P}$ in $\mathsf{Op}(\spectra)$, $\mathbf{S}$ is the sphere spectrum, and $\Map(-,-)$ denotes the mapping spectrum construction. The operad composition maps for $K\mathbf{P}$ are induced by the cooperad structure on $B\mathbf{P}'$.
\end{definition}

\begin{remark}
Although referred to as the `Koszul' dual, the operad $K\mathbf{P}$ is better viewed as the analogue of Ginzburg-Kapranov's `dg-dual'~\cite[\S3]{ginzburg/kapranov:1994} since its construction and properties do not depend on any `Koszulity' property of the operad $\mathbf{P}$.
\end{remark}

\begin{example}
Let $\mathbf{Com}$ be the commutative operad of spectra, given by $\mathbf{Com}(I) = \mathbf{S}$, the sphere spectrum, for all $I$. Then it is shown in~\cite{ching:2005} that $K\mathbf{Com}$ is the spectral Lie-operad, a model for the Goodwillie derivatives of the identity functor on pointed spaces.
\end{example}

The following analogue to \cite[3.2.16]{ginzburg/kapranov:1994} is deduced from Theorem~\ref{thm:bar-cobar} in~\cite[4.11]{ching:2012}

\begin{theorem} \label{thm:koszul}
Let $\mathbf{P}$ be a termwise-finite operad of spectra. Then there is an equivalence of operads
\[ KK\mathbf{P} \homeq \mathbf{P}. \]
\end{theorem}

We can now state precisely the main result of this paper: an identification of the Koszul dual of the \emph{stable} little-disc operad.

\begin{definition}
Let $V$ be a finite-dimensional real normed vector space, and $E_V$ the little $V$-discs operad of Definition~\ref{prop:E}. Let $\mathbf{E}_V$ be the $O(V)$-operad of spectra given by
\[ \mathbf{E}_V := \Sigma^\infty_+ E_V. \]
\end{definition}

\begin{theorem} \label{thm:main}
Let $V$ be a finite-dimensional real normed vector space. Then there is an $O(V)$-equivariant equivalence of operads of spectra
\[ K\mathbf{E}_V \homeq \Sigma^{-V}\mathbf{E}_V. \]
\end{theorem}

\section{The duality map} \label{sec:map}

Our goal in this section is to build the map of operads that underlies our proof of Theorem~\ref{thm:main}. The basic construction is via a collection of $O(V)$-equivariant maps (of pointed spaces)
\begin{equation} \label{eq:map} F_V(I)_+ \smsh BD_V(I) \to \bar{S}_V(I) \end{equation}
where $F_V$ is the Fulton-MacPherson operad of Definition~\ref{def:FM}, $D_V$ is the restricted little disc operad of Definition~\ref{def:D} (with $BD_V$ its bar-cooperad) and $\bar{S}_V$ is the quotient of the $V$-sphere operad given in Definition~\ref{def:barS}. By adjunction, the map (\ref{eq:map}) determines a map of spectra
\[ \Sigma^\infty F_V(I)_+ \to \Map(BD_V(I),\Sigma^\infty \bar{S}_V(I)) \]
which together form a map of operads
\[ \Sigma^\infty F_V \to \Map(BD_V,\Sigma^\infty \bar{S}_V) \]
which we will show to be an equivalence. Theorem~\ref{thm:main} then follows via the following zigzag of equivalences of operads
\[ \begin{split} K\mathbf{E}_V &\weq K\Sigma^\infty_+ D_V \\
     &\weq \Sigma^{-V}\Map(BD_V,\Sigma^\infty S_V) \\
     &\weq \Sigma^{-V}\Map(BD_V,\Sigma^\infty \bar{S}_V) \\
     &\lweq \Sigma^{-V}\Sigma^\infty_+ F_V \\
     &\weq \Sigma^{-V}\mathbf{E}_V. \end{split} \]

In order to describe the map (\ref{eq:map}), we need to be more explicit about the definition of the bar construction $BD_V$. As previewed in~\ref{def:bar}, one way to do this is in terms of certain posets of rooted trees, which we now introduce in more detail.

A \emph{rooted tree} is a contractible one-dimensional finite cell complex with a chosen vertex (the \emph{root}). The choice of root determines an orientation on each edge: towards the root. Each vertex (apart from the root) has a unique outgoing edge, and a possibly empty set of incoming edges. Vertices with no incoming edges are \emph{leaves}, and we are concerned with trees in which the leaves are labelled by the elements of a given finite set $I$. For example, the following diagram illustrates a tree with leaves labelled by the set $\{1,2,3,4,5\}$.
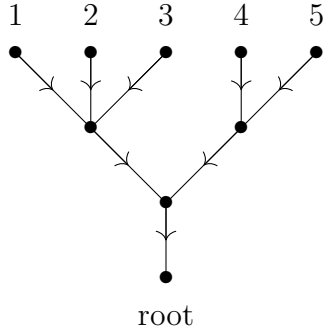
\begin{center}
\begin{tikzpicture}
	\begin{pgfonlayer}{nodelayer}
		\node [style=none] (0) at (0, 0) {$\bullet$};
		\node [style=none] (1) at (-1, 1) {$\bullet$};
		\node [style=none] (2) at (-2, 2) {$\bullet$};
		\node [style=none] (3) at (-1, 2) {$\bullet$};
		\node [style=none] (4) at (0, 2) {$\bullet$};
		\node [style=none] (5) at (1, 2) {$\bullet$};
		\node [style=none] (6) at (2, 2) {$\bullet$};
		\node [style=none] (7) at (1, 1) {$\bullet$};
		\node [style=none] (8) at (0,-1) {$\bullet$};
		\node [style=none] (20) at (0, -1.5) {root};
		\node [style=none] (21) at (-2, 2.5) {$1$};
		\node [style=none] (22) at (-1, 2.5) {$2$};
		\node [style=none] (23) at (0, 2.5) {$3$};
		\node [style=none] (24) at (1, 2.5) {$4$};
		\node [style=none] (25) at (2, 2.5) {$5$};
	\end{pgfonlayer}
	\begin{pgfonlayer}{edgelayer}
		\draw[-{>[scale=1.5]}] (-1,1) to (-0.5,0.5);
		\draw[-{>[scale=1.5]}] (-2,2) to (-1.5,1.5);
		\draw[-{>[scale=1.5]}] (-1,2) to (-1,1.5);
		\draw[-{>[scale=1.5]}] (0,2) to (-0.5,1.5);
		\draw[-{>[scale=1.5]}] (1,1) to (0.5,0.5);
		\draw[-{>[scale=1.5]}] (1,2) to (1,1.5);
		\draw[-{>[scale=1.5]}] (2,2) to (1.5,1.5);
		\draw[-{>[scale=1.5]}] (0,0) to (0,-0.5);
		\draw (2.center) to (1.center);
		\draw (1.center) to (0.center);
		\draw (3.center) to (1.center);
		\draw (4.center) to (1.center);
		\draw (7.center) to (0.center);
		\draw (5.center) to (7.center);
		\draw (6.center) to (7.center);
		\draw (0.center) to (8.center);
	\end{pgfonlayer}
\end{tikzpicture}
\end{center}

We make two further assumptions: that the root vertex has exactly one incoming edge (the \emph{root edge}), and that no other vertex has exactly one incoming edge. It follows that each edge of a tree corresponds uniquely to a certain subset of the labelling set $I$, namely the set of leaves whose paths to the root go via that edge. For example, the eight edges in the tree above correspond to the following subsets of $\{1,2,3,4,5\}$:
\begin{itemize}
    \item the five singletons (the leaf edges);
    \item $\{1,2,3\}$ and $\{4,5\}$ (the internal edges);
    \item the whole set $\{1,2,3,4,5\}$ (the root edge).
\end{itemize}
Moreover, notice that this collection of subsets uniquely determines the structure of the tree. In the rest of this paper, it will be convenient if we \emph{define} trees via that collection of subsets of $I$. We therefore make the following definition.

\begin{definition} \label{def:tree}
Let $I$ be a finite set with $|I| \geq 2$. An \emph{$I$-labelled tree} is a collection $T$ of nonempty subsets of $I$ with the following properties:
\begin{itemize}
  \item each singleton $\{i\}$ is in $T$;
  \item the set $I$ itself is in $T$;
  \item if $e,e' \in  T$, then either $e \subseteq e'$, $e' \subseteq e$, or $e \cap e' = \varnothing$.
\end{itemize}
We refer to the elements of $T$ as the \emph{edges} of the tree $T$. (In this formulation we will not usually refer to the \emph{vertices} of a tree.) The singleton subsets are the \emph{leaves} and the set $I$ itself is the \emph{root}. Each non-root edge $e$ has a unique \emph{outgoing edge} $e'$ that is minimal subject to the condition $e \subsetneq e'$. Each non-leaf edge $e$ has at least two \emph{incoming edges}, i.e.\ those for which $e$ is the outgoing edge. We will denote by $E(T)$ the set of all non-leaf edges of $T$.

A \emph{morphism} from an $I$-labelled tree $T$ to an $I'$-labelled tree $T'$ consists of a bijection $f: I \isom I'$ with the property that for each edge $e \in T$, $f(e) \in T'$. For example, if $f$ is the identity function on a set $I$, then $f$ is a morphism $T \to T'$ if and only if $T \subseteq T'$ (as sets of subsets of $I$). We often therefore visualize such a morphism as given by the insertion of a collection of internal (i.e.\ non-root/leaf) edges.
\end{definition}

\begin{proposition}
There is a small category $\mathsf{Tree}$ whose objects are all the trees in the sense of Definition~\ref{def:tree} and whose morphisms are those defined above. Composition is given by composition of bijections, and the identity morphism on $T$ is the identity bijection on the set of labels of $T$. For each finite set $I$ with $|I| \geq 2$, let $\mathsf{Tree}_I$ denote the subcategory of $\mathsf{Tree}$ whose objects are the $I$-labelled trees, and whose morphisms are those morphisms of $\mathsf{Tree}$ whose underlying bijection is the identity on $I$. In other words, $\mathsf{Tree}_I$ is the poset of $I$-labelled trees with ordering given by inclusion of subsets.
\end{proposition}

\begin{example}
For a nonempty finite set $I$, the \emph{$I$-labelled corolla} is the tree $\tau_I$ consisting only of the singletons and the set $I$ itself.
\end{example}

\begin{definition} \label{def:grafting}
Let $T$ be an $I$-labelled tree and $T'$ a $J$-labelled tree, and take $i \in I$. We then define an $I \cup_i J$-labelled tree $T \cup_i T'$ by saying that the edges in $T \cup_i T'$ are those of one of the following forms:
\begin{itemize}
  \item an edge in $T'$;
  \item an edge in $T$ that does not contain $i$;
  \item $e \cup_i J$ where $e$ is an edge in $T$ that does contain $i$.
\end{itemize}
Note that the edge $J = \{i\} \cup_i J$ is covered twice by these conditions; put another way, we can think of the edges of $T \cup_i T'$ as comprising the edges of $T$ and the edges of $T'$, with the root of $T'$ identified with the leaf $\{i\}$ of $T$.
\end{definition}

\begin{definition} \label{def:w}
Let $T$ be an $I$-labelled tree. We let $\bar{w}(T)$ be the pointed space given by the quotient of the space
\[ [0,\infty]^{E(T)} \]
by the subspace consisting of those sequences $r = (r_e)_{e \in E(T)}$ for which:
\begin{itemize}
  \item $r_e = \infty$ for some $e \in E(T)$; or
  \item $r_I = 0$.
\end{itemize}
For a morphism of $I$-labelled trees given by an inclusion $\iota: T \subseteq T'$, we have a map of pointed spaces $\iota_*: \bar{w}(T) \to \bar{w}(T')$, given by setting $r_e = 0$ for $e \in T' \setminus T$. For an $I$-labelled tree $T$, $i \in I$, and $J$-labelled tree $T'$, we have a map of pointed spaces
\[ \circ^i: \bar{w}(T \cup_i T') \to \bar{w}(T) \smsh \bar{w}(T') \]
given by the identification of each non-leaf edge in $T \cup_i T'$ with a non-leaf edge either in $T$ or $T'$.
\end{definition}

\begin{prop} \label{prop:w}
For each finite set $I$, the maps $\iota_*$ described in Definition~\ref{def:w} together form a functor $\bar{w}: \mathsf{Tree}_I \to \based$. The maps $\circ^i$ are natural with respect to $T \in \mathsf{Tree}_I$ and $T' \in \mathsf{Tree}_J$, and also associative with respect to multiple grafting maps.
\end{prop}

\begin{definition} \label{def:bar1}
Let $P$ be an operad of either pointed spaces or spectra. For an $I$-labelled tree $T$, we define
\[ P(T) := \Smsh_{e \in E(T)} P(I_e) \]
where $I_e$ denotes the set of incoming edges of $T$ to a non-leaf edge $e$. The composition maps for $P$ determine a functor $P(-): \mathsf{Tree}_I^{op} \to \based$.

The \emph{bar-cooperad} $BP$ is then given by the coend
\[ BP(I) := \bar{w}(T) \smsh_{T \in \mathsf{Tree}_I} P(T) \]
with cooperad structure maps $BP(I \cup_i J) \to BP(I) \smsh BP(J)$ induced by the maps $\circ^i$ of Definition~\ref{def:w} together with the isomorphism
\[ P(T \cup_i T') \isom P(T) \smsh P(T') \]
given again by identifying non-leaf edges of $T \cup_i T'$ with those of $T$ and $T'$.
\end{definition}

\begin{prop}[\cite{salvatore:1999},\cite{ching:2005}]
Let $P$ be an operad of pointed spaces or spectra. The structure maps of Definition~\ref{def:bar1} make $BP$ into a cooperad of pointed spaces or spectra respectively. We refer to $BP$ as the \emph{bar-cooperad} of $P$.
\end{prop}

\begin{remark}
For an operad $P$ of \emph{unpointed} spaces, we write $BP$ for the bar-cooperad of $P_+$, the operad of pointed spaces obtained by adding a disjoint basepoint to each term of $P$. In particular, this notation applies when $P = D_V$ the restricted little disc operad.
\end{remark}

We now turn to the construction of a map
\[ \alpha: F_V(I)_+ \smsh BD_V(I) \to \bar{S}_V(I). \]
According to Definition~\ref{def:bar1}, such a map $\alpha$ will be determined by a suitable collection of maps
\[ \alpha_T: F_V(I)_+ \smsh \bar{w}(T) \smsh D_V(T)_+ \to \bar{S}_V(I) \]
for each $I$-labelled tree $T$, where
\[ D_V(T) := \prod_{e \in E(T)} D_V(I_e). \]
The operad composition maps for $D_V$ allow us to identify $D_V(T)$ with a closed subspace of $D_V(I)$ which we now describe.

\begin{lemma} \label{lem:DV}
For an $I$-labelled tree $T$, the operad composition maps for $D_V$ determine a closed embedding
\[ D_V(T) = \prod_{e \in E(T)} D_V(I_e) \to D_V(I) \subseteq R_V(I) \]
whose image is the set of points $(x,t) \in R_V(I)$ such that
\begin{itemize}
  \item for edges $e' \subseteq e$ in $T$, $|x_{e'} - x_e| \leq t_e - t_{e'}$;
  \item for edges $e'' \cap e = \varnothing$ in $T$, $|x_{e''} - x_e| \geq t_e + t_{e''}$.
\end{itemize}
\end{lemma}
\begin{proof}
Let $D'_V(T)$ be the subset of $R_V(T)$ of points $(x,t)$ satisfying the two given conditions. We prove that $D'_V(T) = D_V(T)$ by induction on the number of non-leaf edges in $T$.

When $T$ is a corolla, the first condition (applied when $e$ is the root and $e'$ is the leaf labelled $i$), says that $|x_i| \leq 1 - t_i$, and the second condition (applied to leaves labelled $i$ and $i'$) says that $|x_i - x_{i'}| \geq t_i + t_{i'}$. Thus in this case $D'_V(T) = D_V(I) = D_V(T)$.

It is now sufficient to show that $(x,t) \in D'_V(T \cup_i T')$ if and only if $(x/J,t/J) \in D'_V(T)$ and $(x|J,t|J) \in D'_V(T')$, with $T$ and $T'$ respectively $I$-labelled and $J$-labelled trees, and $i \in I$. For a given edge $e$ in $T$, we denote by $\bar{e}$ the corresponding edge in $T \cup_i T'$. Suppose first that $(x,t) \in D'_V(T \cup_i T')$. If $e' \subseteq e$ in $T$, then by Lemma~\ref{lem:btcalc}:
\[ |(x/J)_{e'} - (x/J)_e| = |x_{\bar{e}'} - x_{\bar{e}}| \leq t_{\bar{e}'} - t_{\bar{e}} = (t/J)_{e'} - (t/J)_e, \]
and if $e'' \cap e = \varnothing$ in $T$, then
\[ |(x/J)_{e''} - (x/J)_e| = |x_{\bar{e}''} - x_{\bar{e}}| \geq t_{\bar{e}''} + t_{\bar{e}} = (t/J)_{e''} + (t/J)_e, \]
so that $(x/J,t/J) \in D'_V(T)$. Similarly, if $e' \subseteq e$ in $T'$, then
\[ |(x|J)_{e'} - (x|J)_e| = t_J^{-1}|x_{e'} - x_e| \leq t_J^{-1}(t_{e'} - t_{e}) = (t|J)_{e'} - (t|J)_e, \]
and if $e'' \cap e = \varnothing$ in $T'$, then
\[ |(x|J)_{e''} - (x|J)_e| = t_J^{-1}|x_{e''} - x_e| \geq t_J^{-1}(t_{e''} + t_{e}) = (t|J)_{e''} + (t|J)_e, \]
so that $(x/J,t/J) \in D'_V(T')$.

Conversely, suppose that $(x/J,t/J) \in D'_V(T)$ and $(x|J,t|J) \in D'_V(T')$, and consider two nested edges in $T \cup_i T'$. If both edges are in $T$, or both in $T'$, then very similar calculations to those above imply the desired inequality. So assume the edges are $e' \subseteq \bar{e}$ for $e \in T$ and $e' \in T'$. Notice that $\bar{\{i\}} = J$, i.e.\ the leaf edge of $T$ corresponds to the root edge of $T'$ inside $T \cup_i T'$. We then have
\[ |x_{\bar{e}} - x_{\bar{\{i\}}}| = |(x/J)_e - (x/J)_i| \leq (t/J)_e - (t/J)_i = t_{\bar{e}} - t_J \]
and
\[ |x_J - x_{e'}| = t_J|(x|J)_J - (x|J)_{e'}| \leq t_J(t|J)_J - t_J(t|J)_{e'} = t_J - t_{e'} \]
and so
\[ |x_{\bar{e}} - x_{e'}| \leq t_{\bar{e}} - t_{e'} \]
as desired. Similarly, suppose that $\bar{e}$ and $e''$ are disjoint edges of $T \cup_i T'$, where $e$ is an edge of $T$ and $e''$ and edge of $T'$. Then $e$ is disjoint from $\{i\}$ in $T$ and so
\[ |x_{\bar{e}} - x_{\bar{\{i\}}}| = |(x/J)_e - (x/J)_i| \geq (t/J)_e + (t/J)_i = t_{\bar{e}} + t_J \]
and $e'$ is contained in $J$, so $|x_J - x_{e'}| \leq t_J - t_{e'}$ as above. Therefore
\[ |x_{\bar{e}} - x_{e'}| \geq |x_{\bar{e}} - x_{\bar{\{i\}}}| - |x_J - x_{e'}| \geq t_{\bar{e}} + t_{e'} \]
as desired. Therefore, $(x,t) \in D'_V(T \cup_i T')$.
\end{proof}

\begin{definition} \label{def:alpha}
Let $T$ be an $I$-labelled tree, and take points $y \in F_V(I)$, $r \in [0,\infty]^{E(T)}$ and $(z,t) \in D_V(T)$. First suppose that $r_e \neq \infty$ for all $e \in E(T)$. We define
\[ \alpha_T(y,r,(z,t)) := (x,t) \in R_V(I) \]
by setting
\[ x_i := z_i - \sum_{i \in e \in E(T)} t_e r_e y(e)_i \]
where we are assuming that the tuple $y(e) \in V^e$ satisfies the weighted barycentre and weighted norm conditions of Remark~\ref{rem:FM}. The sum here is taken over all non-leaf edges $e$ of $T$ that contain $i$. Notice that the pair $(x,t)$ is indeed in $R_V(I)$ as required.

We extend $\alpha_T$ to a function
\begin{equation} \label{eq:alpha} \alpha_T: F_V(I) \times [0,\infty]^{E(T)} \times D_V(T) \to S_V(I) \end{equation}
by setting $\alpha_T(y,r,(z,t)) := \infty$ (i.e.\ the basepoint in $S_V(I)$) if $r_e = \infty$ for some edge $e$.
\end{definition}

\begin{lemma}
The map $\alpha_T$ of (\ref{eq:alpha}) is continuous.
\end{lemma}
\begin{proof}
Choose some point $(y,r,(z,t))$ where $r_e = \infty$ for some $e \in E(T)$, and consider a sequence of points
\[ (y^{(n)},r^{(n)},(z^{(n)},t^{(n)})) \in F_V(I) \times [0,\infty]^{E(T)} \times D_V(T) \]
that converges to $(y,r,(z,t))$. Without loss of generality we may assume that $r^{(n)}_e \neq \infty$ for all $n \in \mathbb{N}$ and $e \in E(T)$.
We then claim that
\[ (x^{(n)},t^{(n)}) := \alpha_T(y^{(n)},r^{(n)},(z^{(n)},t^{(n)})) \to \infty \in S_V(I). \]
To see this choose an edge $e \in E(T)$ minimal (i.e.\ furthest from the root) such that $r_e = \infty$. The vectors in the $e$-tuple $y(e)$ are not all equal, so we can choose $i,j \in e$ such that
\[ y(e)_i - y(e)_j \neq 0. \]
and $e$ is the smallest
edge containing both $i$ and $j$. Now consider the sequence of points in $R_V(I)$ given by
\[ x^{(n)}_i - x^{(n)}_j = z^{(n)}_i - z^{(n)}_j - \sum_{i \in e'} t^{(n)}_{e'} r^{(n)}_{e'} y^{(n)}(e')_i + \sum_{j \in e'} t^{(n)}_{e'} r^{(n)}_{e'} y^{(n)}(e')_j. \]
Since $i,j \in e$, we can write this expression as
\[ z^{(n)}_i - z^{(n)}_j - \sum_{i \in e' \subsetneq e} t^{(n)}_{e'} r^{(n)}_{e'} y^{(n)}(e')_i + \sum_{j \in e' \subsetneq e}
t^{(n)}_{e'} r^{(n)}_{e'} y^{(n)}(e')_j - \sum_{e \subseteq e'} t^{(n)}_{e'} r^{(n)}_{e'} (y^{(n)}(e')_i - y^{(n)}(e')_j) \]
where each sum is over all non-leaf edges $e' \in E(T)$ satisfying the given condition.

If we write $v^{(n)}$ for the sum of the first four terms in the above sum, then by the minimality of $e$, the sequence $(v^{(n)})$ converges to the finite vector
\[ v := z_i - z_j - \sum_{i \in e' \subsetneq e} t_{e'}r_{e'}y(e')_i + \sum_{j \in e' \subsetneq e} t_{e'}r_{e'}y(e')_j. \]
The key observation then is that each of the vectors
\[ y^{(n)}(e')_i - y^{(n)}(e')_j \]
for $e \subseteq e'$, is a non-negative scalar multiple $\lambda^{(n)}_{e'}$ of $y^{(n)}(e)_i - y^{(n)}(e)_j$. This comes from the restriction condition $y^{(n)}(e')|_e \equiv y^{(n)}(e)$, modulo translation and scaling, on the tuples of vectors making up the point $y^{(n)} \in F_V(I)$, as in Definition~\ref{def:FM}.

We can therefore write
\[ x^{(n)}_i - x^{(n)}_j = v^{(n)} + \left( \sum_{e \subseteq e'} t^{(n)}_{e'} r^{(n)}_{e'} \lambda^{(n)}_{e'} \right) (y^{(n)}(e)_i - y^{(n)}(e)_j). \]
In this sum, we have $(r^{(n)}_e) \to r_e = \infty$, $(\lambda^{(n)}_e) \to 1$ and $(t^{(n)}_e) \to t_e > 0$. It follows that the sum converges to $\infty$. Since $y^{(n)}(e)_i - y^{(n)}(e)_j \to y(e)_i - y(e)_j \neq 0$, it follows that
\[ (x^{(n)}_i - x^{(n)}_j) \to \infty \]
in the one-point compactification of $V$. This observation implies that $(x^{(n)},t^{(n)}) \to \infty$ in $S_V(I)$ as required, thus establishing the continuity of the map $\alpha_T$ of (\ref{eq:alpha}).
\end{proof}

We next note the following easy consequence of the definition of $\alpha_T$ in \ref{def:alpha}.

\begin{lemma} \label{lem:be}
Suppose $r_e \neq \infty$ for all $e \in E(T)$. Let $e$ be an edge of $T$. Then we have
\[ \alpha_T(y,r,(z,t))_e = z_e - \sum_{e \subsetneq e'} t_{e'} r_{e'} y(e')_e. \]
\end{lemma}
\begin{proof}
By definition, the left-hand side is equal to
\[ \frac{1}{t_e}\sum_{i \in e} t_i (z_i - \sum_{i \in e'} t_{e'}r_{e'}y(e')_i) = z_e - \frac{1}{t_e}\sum_{e' \in T} t_{e'}r_{e'}\sum_{i \in e \cap e'}t_i y(e')_i = z_e - \frac{1}{t_e}\sum_{e' \in T} t_{e'}r_{e'} t_{e \cap e'}y(e')_{e \cap e'}. \]
If $e'$ and $e$ are edges in the same tree with nonempty intersection, we have either $e' \subseteq e$ (in which case $y(e')_{e \cap e'} = y(e')_{e'} = 0$ so these terms vanish) or $e \subsetneq e'$. Thus we obtain the desired formula.
\end{proof}

\begin{lemma}
The map $\alpha_T$ of (\ref{eq:alpha}) induces a well-defined map of pointed spaces
\[ \alpha_T: F_V(I)_+ \smsh \bar{w}(T) \smsh D_V(T)_+ \to \bar{S}_V(I). \]
\end{lemma}
\begin{proof}
From the definition, we already know that if $r_e = \infty$ for some $e \in E(T)$, then $\alpha_T(y,r,(z,t)) = \infty$ in $S_V(I)$. It remains to consider the case where $r_I = 0$. Let $e,e''$ be two incoming edges to the root in $T$. Then, if $(x,t) = \alpha_T(y,r,(z,t))$, we have
\[ |x_e - x_{e''}| = |(z_e - z_{e''}) - (t_I r_I y(I)_e - t_I r_I y(I)_{e''})| = |z_e - z_{e''}| \geq t_e + t_{e''} \]
by Lemma~\ref{lem:DV} and 
Lemma~\ref{lem:be}. But then it is not the case that $|x_e - x_{e''}| < \min\{t_e,t_{e''}\}$ and so $(x,t) \notin \mathring{S}_V(I)$ by Lemma~\ref{lem:barS}. Therefore $(x,t)$ is the basepoint in $\bar{S}_V(I)$.
\end{proof}

\begin{lemma} \label{lem:alpha-map}
For each finite set $I$, the maps $\alpha_T$ for all $I$-labelled trees $T$, induce a map
\[ \alpha_I: F_V(I)_+ \smsh BD_V(I) = F_V(I)_+ \smsh \bar{w}(T) \smsh_{T \in \mathsf{Tree}_I} D_V(T)_+ \to \bar{S}_V(I). \]
\end{lemma}
\begin{proof}
We have to show that the maps $\alpha_T$ are compatible with the structure maps for the functors $\bar{w}: \mathsf{Tree}_I \to \based$ and $D_V(-)_+: \mathsf{Tree}_I^{op} \to \based$, i.e.\ that given $T \subseteq T'$ in $\mathsf{Tree}_I$, the following diagram commutes:
\[ \begin{diagram}
  \node{F_V(I)_+ \smsh \bar{w}(T) \smsh D_V(T')_+} \arrow{s} \arrow{e} \node{F_V(I)_+ \smsh \bar{w}(T) \smsh D_V(T)_+} \arrow{s,r}{\alpha_T} \\
  \node{F_V(I)_+ \smsh \bar{w}(T') \smsh D_V(T')_+} \arrow{e,t}{\alpha_{T'}} \node{\bar{S}_V(I)}
\end{diagram} \]
Note that since we have identified $D_V(T)$ and $D_V(T')$ with subspaces of $D_V(I)$, as in Lemma~\ref{lem:DV}, the operad composition map $D_V(T') \to D_V(T)$ is simply the inclusion between these two subspaces.

So take $(y,r,(z,t)) \in F_V(I) \times [0,\infty]^{E(T)} \times D_V(T')$. The difference between the two composites in the above diagram consists of terms of the form
\[ t_{e'}(\iota_*r)_{e'}y(e')_i \]
where $e' \in E(T') - E(T)$, and $\iota_*$ is as in Definition~\ref{def:w}. But $(\iota_*r)_{e'} = 0$ in each such case, so the diagram commutes.
\end{proof}

We have now constructed the desired map (of symmetric sequences) $F_V \smsh BD_V \to \bar{S}_V$. The next lemma shows that this map respects the operad and cooperad structures on $F_V$, $\bar{S}_V$ and $BD_V$. Its proof is the most challenging and technical part of the paper.

\begin{lemma} \label{lem:alpha-operad}
For finite sets $I,J$ and $i \in I$, there is a commutative diagram
\[ \begin{diagram}
  \node{F_V(I)_+ \smsh F_V(J)_+ \smsh BD_V(I \cup_i J)} \arrow{e,r}{\circ_i(F_V)} \arrow{s,l}{\circ^i(BD_V)} \node{F_V(I \cup_i J)_+ \smsh BD_V(I \cup_i J)} \arrow[2]{s,r}{\alpha_{I \cup_i J}} \\
  \node{F_V(I)_+ \smsh F_V(J)_+ \smsh BD_V(I) \smsh BD_V(J)} \arrow{s,l}{\alpha_I \smsh \alpha_J} \\
  \node{\bar{S}_V(I) \smsh \bar{S}_V(J)} \arrow{e,t}{\circ_i(\bar{S}_V)} \node{\bar{S}_V(I \cup_i J)}
\end{diagram} \]
\end{lemma}
\begin{proof}
Choose points $y' \in F_V(I)$, $y'' \in F_V(J)$, an $(I \cup_i J)$-labelled tree $T$, $r \in [0,\infty]^{E(T)}$ and $(z,t) \in D_V(T)$.

Suppose first that $T = T' \cup_i T''$ for an $I$-labelled trees $T'$ and a $J$-labelled tree $T''$. Then we can write $r = (r',r'')$ for $r' \in [0,\infty]^{E(T')}, r'' \in [0,\infty]^{E(T'')}$, and we have $(z,t) = (z',t') \circ_i (z'',t'')$ for $(z',t') \in D_V(T'), (z'',t'') \in D_V(T'')$, that is: $z = z' +_i t'z'', t = t' \cdot_i t''$.

Going clockwise around the diagram in question, we obtain the point
\[ (x,t) = \alpha_{T' \cup_i T''}(y' \circ_i y'', (r',r''), (z'+_it'z'',t' \cdot_i t'')) \]
which has the following components: for $i' \in I - \{i\}$:
\[ x_{i'} = z'_{i'} - \sum_{i' \in e \in T'} (t' \cdot_i t'')_{\bar{e}} r'_e (y' \circ_i y'')(\bar{e})_{i'} \]
where $\bar{e} = e$ if $i \notin e$ and $\bar{e} = e \cup_i J$ if $i \in e$. Either way, we have
\[ (t' \cdot_i t'')_{\bar{e}} = t'_e \]
and, according to Definition~\ref{def:FM-operad}, we have
\[ (y' \circ_i y'')(\bar{e})_{i'} = y'(e)_{i'}. \]
We therefore conclude
\[ x_{i'} = \alpha_{T'}(y',r',(z',t'))_{i'}. \]
For $j \in J$:
\[ x_j = (z'_i + t'_iz''_j) - \sum_{j \in e \in E(T'')} (t' \cdot_i t'')_e r''_e y''(e)_j - \sum_{i \in e \in E(T')} (t' \cdot_i t'')_{e \cup_i J} r'_e (y' \circ_i y'')(e \cup_i J)_j. \]
In the first sum, we have
\[ (t' \cdot_i t'')_e = t'_it''_e, \]
and in the second we still have $(t' \cdot_i t'')_{e \cup_i J} = t'_e$ and
\[ (y' \circ_i y'')(e \cup_i J)_j = y'(e)_i. \]
Altogether this gives us
\[ x_j = \alpha_{T'}(y',r',(z',t'))_i + t'_i \alpha_{T''}(y'',r'',(z'',t''))_j. \]
In other words, we have
\begin{equation} \label{eq:alphaTT} \alpha_{T' \cup_i T''}(y' \circ_i y'', (r',r''), (z'+_it'z'',t' \cdot_i t'')) = \alpha_{T'}(y',r',(z',t')) +_i t'\alpha_{T''}(y'',r'',(z'',t'')) \end{equation}
It is easy to check that going anticlockwise around the diagram in question yields the right-hand side of this equation.

Now suppose that $T$ is not of the form $T' \cup_i T''$, i.e.\ $J$ is not an edge in the tree $T$. Applying the $BD_V$ cooperad decomposition map $\circ^i$ to such a point in $BD_V(I \cup_i J)$ yields the basepoint in $BD_V(I) \smsh BD_V(J)$. Thus going anticlockwise around the diagram we obtain the basepoint in $\bar{S}_V(I \cup_i J)$. We must therefore show that
\[ (x,t) = \alpha_T(y,r,(z,t)) \]
is also the basepoint, whenever $y$ is in the image of the operad composition map
\[ \circ_i: F_V(I) \times F_V(J) \to F_V(I \cup_i J) \]
and $J$ is not an edge in the tree $T$.

Suppose first that $y$ is also in the image of the operad composition map for $F_V$ associated to some subset of $I \cup_i J$ that \emph{is} an edge in $T$. Then the argument of the first part of this proof, specifically equation (\ref{eq:alphaTT}), allows us to reduce to the case that $y$ is \emph{not} in the image of the operad composition map associated to any edge of $T$. We can then also assume without loss of generality that $J$ is a maximal subset of $I \cup_i J$ with the property that $y$ is in that image. This assumption means in particular that the tuple $y(I \cup_i J)|_J$ of vectors in $V$ is constant, but that $y(I \cup_i J)|_K$ is not constant for any $K \subseteq I \cup_i J$ that properly contains $J$.


Now suppose that $(x,t) \in \mathring{S}_V(I \cup_i J)$, i.e.\ is not the basepoint in $\bar{S}_V(I \cup_i J)$. We will develop some consequences of this assumption for the barycentres of the point $x$ with respect to a sequence of consecutive edges
\[ e_m \subsetneq e_{m-1} \subsetneq \dots \subsetneq e_1 \subsetneq e_0 \]
of the $I \cup_i J$-labelled tree $T$, where $e_m \cap J \neq \varnothing$ and $e_{m-1} \nsubseteq J$.

Let us write $b = z - x$. Then, for $q = 1,\dots,m$:
\[ |b_{e_q} - z_{e_1}| \leq |x_{e_q}| + |z_{e_q} - z_{e_1}| < t_{e_q} + (t_{e_1}-t_{e_q}) = t_{e_1}. \]
by definition of $\mathring{S}_V$ and by Lemma \ref{lem:DV}.
It follows that any convex combination
\[ v = \alpha_1 b_{e_1} + \dots + \alpha_m b_{e_m} \]
also satisfies
\begin{equation} \label{eq:vbz} |v - z_{e_1}| < t_{e_1}. \end{equation}
Recall from Lemma~\ref{lem:be} that we have
\begin{equation} \label{eq:be} b_{e_q} = t_{e_{q-1}}r_{e_{q-1}}y(e_{q-1})_{e_q} + \dots + t_{e_1}r_{e_1}y(e_1)_{e_q} + v_q \end{equation}
where
\[ v_q := \sum_{e_0 \subseteq d} t_d r_d y(d)_{e_q}. \]
We claim that, $v_m$ is a convex combination of $b_{e_1},\dots,b_{e_m}$ and hence is subject to the condition (\ref{eq:vbz}). Notice that $v_1 = b_{e_1}$ by definition, so let us proceed by induction. Suppose that $v_q$ is a convex combination of $b_{e_1},\dots,b_{e_q}$ for all $q < m$.

We will describe the relationship between the terms $y(e_p)_{e_q}$, with $p<q$, that appear in (\ref{eq:be}) and the vectors $v_q$. The key to this relationship is the connection between the different tuples $y(e_p)$ of vectors in $V$ that make up the point $y \in F_V(I \cup_i J)$. For each $p < m$ and each edge $d$ of $T$ such that $e_0 \subseteq d$, we have
\[ y(d)|_{e_p} \equiv y(e_p) \]
modulo translation and scaling by a non-negative constant $\gamma_{p,d}$, so that
\[ y(d)_{e_q} - y(d)_{e_p} = \gamma_{p,d}(y(e_p)_{e_q} - y(e_p)_{e_p}) = \gamma_{p,d} y(e_p)_{e_q}. \]
Recalling our maximality assumption for $J$, we know that $y(I \cup_i J)|_{J}$ is a constant tuple of vectors in $V$, but that $y(I \cup_i J)|_{e_p \cup J}$ is not (since $e_p \nsubseteq J$). Since $e_p \cap J \neq \varnothing$, it follows that $y(I \cup_i J)|_{e_p}$ is not a constant tuple, and so $y(d)|_{e_p}$ is not constant either. Therefore we have $\gamma_{p,d} > 0$.

It then follows that
\[ v_q - v_p = \gamma_p y(e_p)_{e_q} \]
where $\gamma_p = \sum_{e_0 \subseteq d} t_d r_d \gamma_{p,d}$. If the values $r_d$ for $e_0 \subseteq d$ were all equal to $0$, then we would have $v_1,\dots,v_m = 0$ and so in particular $v_m$ would be a convex combination of $b_{e_1} = v_1 = 0$. So we may assume that not all $r_d$ are equal to $0$, and hence that $\gamma_p > 0$.

From (\ref{eq:be}) we now have
\[ b_{e_m} = \lambda_{m-1}(v_m-v_{m-1}) + \lambda_{m-2}(v_m - v_{m-2}) + \dots + \lambda_1(v_m-v_1) + v_m \]
where
\[ \lambda_p = \frac{t_{e_p}r_{e_p}}{\gamma_p} \geq 0. \]
Our induction hypothesis is that
\[ v_q = \alpha_{1,q}b_{e_1} + \dots + \alpha_{q,q}b_{e_q} \]
where $\alpha_{p,q} \geq 0$ and $\alpha_{1,q} + \dots + \alpha_{q,q} = 1$. We therefore have
\[ b_{e_m} = (1+\lambda_1+\dots+\lambda_{m-1})v_m - \sum_{1 \leq p \leq q \leq m-1} \lambda_q\alpha_{p,q}b_{e_p} \]
and so
\[ v_m = \alpha_{1,m}b_{e_1} + \dots + \alpha_{m,m}b_{e_m} \]
where
\[ \alpha_{m,m} = \frac{1}{1+\lambda_1+\dots+\lambda_{m-1}} \]
and
\[ \alpha_{p,m} = \frac{\sum_{p \leq q \leq m-1} \lambda_q\alpha_{p,q}}{1+\lambda_1+\dots+\lambda_{m-1}} \]
for $p = 1,\dots,m-1$. Each of these coefficients is non-negative, and
\[ \alpha_{1,m} + \dots + \alpha_{m,m} = \frac{1+\sum_{1 \leq p \leq q \leq m-1} \lambda_q\alpha_{p,q}}{1+\lambda_1+\dots+\lambda_{m-1}} = 1 \]
as required, completing the induction.

Now return to the tree $T$, and choose distinct edges $e',e'' \in T$ that are maximal with respect to the condition that $e',e'' \subseteq J$. (Since $J$ is not an edge of the tree $T$, there are at least two such edges. Since $T$ is a tree, $e'$ and $e''$ must be disjoint.)

We apply the preceding calculations to the sequences of edges
\[ e' = e'_m \subsetneq e'_{m-1} \subsetneq \dots \subsetneq e'_0 = e \]
and
\[ e'' = e''_l \subsetneq e''_{l-1} \subsetneq \dots \subsetneq e''_0 = e \]
where $e$ is the minimal edge containing both $e'$ and $e''$. Note that $e'_{m-1},e''_{l-1} \nsubseteq J$ by the maximality of $e',e''$.

Our previous analysis now implies that the point
\[ v' := \sum_{e \subseteq d} t_d r_d y(d)_{e'} \]
is a convex combination of $b_{e'_m}, \dots, b_{e'_1}$ and hence satisfies the condition
\[ |v' - z_{e'_1}| < t_{e'_1}. \]
Similarly, the point
\[ v'' = \sum_{e \subseteq d}t_d r_d y(d)_{e''} \]
satisfies the condition
\[ |v'' - z_{e''_1}| < t_{e''_1}. \]
Moreover, by Lemma~\ref{lem:DV} we have
\[ |z_{e''_1} - z_{e'_1}| \geq t_{e'_1}+t_{e''_1} \]
and so we must have
\[ |v' - v''| > 0. \]
However, $y$ is in the image of the composition map
\[ F_V(I) \times F_V(J) \to F_V(I \cup_i J), \]
and $e \nsubseteq J$, so each $d$-tuple $y(d)$, for $e \subseteq d$, has the property that
\[ y(d)|_J \]
is a constant tuple of vectors in $V$. Since $e',e'' \subseteq J$, we have
\[ y(d)_{e'} = y(d)_{e''} \]
for all such $e$, and so $v' = v''$, a contradiction. Thus, in fact, $(x,t) \notin \mathring{S}_V(I \cup_i J)$ and so $\alpha_T(y,r,(z,t))$ is the basepoint in $\bar{S}_V(I \cup_i J)$, as required.
\end{proof}

We can now define the desired map of operads of spectra of the form (\ref{eq:map}).

\begin{definition}
Let $V$ be a finite-dimensional normed vector space, and let $I$ be a finite set with $|I| \geq 2$. Then we define an $O(V)$-equivariant map of spectra
\[ \alpha_I^{\#}: \Sigma^\infty F_V(I)_+ \to \Map(BD_V(I), \Sigma^\infty \bar{S}_V(I)) \]
by applying $\Sigma^\infty$ to the map $\alpha_I$ of Lemma~\ref{lem:alpha-map} and using the adjunction between smash product and mapping spectrum.
\end{definition}

\begin{proposition} \label{prop:alphah}
The maps $\alpha_I^{\#}$ together form a map of operads of spectra
\[ \alpha^{\#}: \Sigma^\infty F_V \to \Map(BD_V,\Sigma^\infty \bar{S}_V). \]
\end{proposition}
\begin{proof}
The operad composition maps on the right-hand side are given by combining the cooperad structure maps for $BD_V$ with the operad structure maps for $\bar{S}_V$. That $\alpha^{\#}$ respects the operad structures is a consequence of Lemma~\ref{lem:alpha-operad}.
\end{proof}

As we have already seen, our main result Theorem~\ref{thm:main} follows from the claim that $\alpha^{\#}$ is an equivalence. To prove this claim, we will apply the bar-cobar duality equivalence of Theorem~\ref{thm:bar-cobar} and show that $\alpha^{\#}$ is adjoint to a certain equivalence of quasi-cooperads. In order to explain this argument, we recall from \cite{ching:2012} the definition of a quasi-cooperad and of the left adjoint $\mathbb{B}$ appearing in Theorem~\ref{thm:bar-cobar}.

\section{Quasi-cooperads and the bar construction} \label{sec:quasi}

We first recall the notion of quasi-cooperad of spectra from \cite{ching:2012}. We start with a `pre-cooperad', a notion based on the full category of trees described in Definition~\ref{def:tree}.

\begin{definition} \label{def:q-coop}
A \emph{pre-cooperad of spectra} $\mathbf{Q}$ consists of
\begin{itemize}
  \item a functor $\mathbf{Q}: \mathsf{Tree} \to \spectra$;
  \item for each $I$-labelled tree $I$, $J$-labelled tree $T'$, and $i \in I$, a \emph{grafting map}
      \[ \mu_i: \mathbf{Q}(T) \smsh \mathbf{Q}(T') \to \mathbf{Q}(T \cup_i T'); \]
\end{itemize}
such that the maps $\mu_i$ are natural (with respect to morphisms in $\mathsf{Tree}$) and suitably associative (with respect to grafting of multiple trees). A \emph{quasi-cooperad} is a pre-cooperad for which the grafting maps $\mu_i$ are weak equivalences of spectra.
\end{definition}

\begin{remark}
Replacing spectra with pointed spaces in Definition~\ref{def:q-coop} we obtain definitions of pre/quasi-cooperad of pointed spaces. More generally, by replacing the smash product, we obtain a notion of pre-cooperad in any symmetric monoidal category. Given a choice of weak equivalences, we get a corresponding notion of quasi-cooperad.
\end{remark}

\begin{example}
If $Q$ is a quasi-cooperad of pointed spaces, then $\Sigma^\infty Q$ is a quasi-cooperad of spectra.
\end{example}

\begin{example} \label{ex:q-coop}
If $\mathbf{Q}$ is a cooperad (of spectra or pointed spaces), we define a corresponding quasi-cooperad by setting
\[ \mathbf{Q}(T) := \Smsh_{e \in E(T)} \mathbf{Q}(I_e) \]
where $I_e$ is the set of incoming edges to a given non-leaf edge $e$ of $T$. The value of the functor $\mathbf{Q}: \mathsf{Tree} \to \spectra$ on morphisms is given by the cooperad decomposition maps for $\mathbf{Q}$. In this case, the grafting maps are isomorphisms.
\end{example}

\begin{example} \label{ex:op-qcoop}
Let $\mathbf{P}$ be an operad for which the composition maps $\mathbf{P}(I) \smsh \mathbf{P}(J) \weq \mathbf{P}(I \cup_i J)$ are weak equivalences. Then we define a corresponding quasi-cooperad by setting
\[ \mathbf{P}(T) := \mathbf{P}(I). \]
The value of the functor $\mathbf{P}: \mathsf{Tree} \to \spectra$ on a morphism $f: I \isom I'$ in $\mathsf{Tree}$ is the induced isomorphism $\mathbf{P}(I) \isom \mathbf{P}(I')$. The grafting maps are given by the composition maps for $\mathbf{P}$.
\end{example}

\begin{example} \label{ex:map-qcoop}
Let $\mathbf{P}$ be an operad and define a quasi-cooperad $\dual \mathbf{P}$ by the collection of Spanier-Whitehead duals
\[ (\dual \mathbf{P})(T) := \Map(\mathbf{P}(T),S^0) \]
where
\[ \mathbf{P}(T) := \Smsh_{e \in E(T)} \mathbf{P}(I_e). \]
More generally, if $\mathbf{P}$ is an operad (of spectra or pointed spaces) and $\mathbf{Q}$ is a quasi-cooperad of spectra, then there is a quasi-cooperad of spectra $\Map(\mathbf{P},\mathbf{Q})$ given by
\[ \Map(\mathbf{P},\mathbf{Q})(T) := \Map(\mathbf{P}(T),\mathbf{Q}(T)). \]
\end{example}

We now turn to the left adjoint functor $\mathbb{B}$ appearing in Theorem~\ref{thm:bar-cobar}. As in most of this section, we can make this construction in the same way for operads of spectra or of pointed spaces (and hence, by adding a disjoint basepoint, unpointed spaces). The definition relies on a variant of the pointed space $\bar{w}(T)$ appearing in the definition of the ordinary bar construction $BP$.

\begin{definition} \label{def:w2}
Suppose $T,U$ are $I$-labelled trees. Then we let $\bar{w}(T;U)$ be the quotient of the space $[0,\infty]^{E(T)}$ by the subspace consisting of those points $r = (r_e)_{e \in E(T)}$ for which
\begin{itemize}
  \item $r_e = \infty$ for any $e \in E(T)$;
  \item $r_e = 0$ for any $e \in E(U)$.
\end{itemize}
Comparing with Definition~\ref{def:w} we see that $\bar{w}(T) = \bar{w}(T;\tau_I)$, with $\tau_I$ the corolla tree. The obvious quotient and inclusion by $0$ maps make up a functor
\[ \bar{w}: \mathsf{Tree}_I \times \mathsf{Tree}_I \to \based \]
and there are isomorphisms
\[ \bar{w}(T,U) \smsh \bar{w}(T',U') \isom \bar{w}(T \cup_i T',U \cup_i U') \]
given by identifying non-leaf edges of $T \cup_i T'$ with those of $T$ and of $T'$, and similarly for $U \cup_i U'$. The degrafting maps of Definition~\ref{def:w} are then the composites
\[ \bar{w}(T \cup_i T',\tau_{I \cup_i J}) \to \bar{w}(T \cup_i T',\tau_I \cup_i \tau_J) \isom \bar{w}(T,\tau_I) \smsh \bar{w}(T';\tau_J). \]
\end{definition}

\begin{definition} \label{def:BP}
Let $\mathbf{P}$ be an operad of spectra or pointed spaces. We then construct a pre-cooperad $\mathbb{B}\mathbf{P}$ as follows. For an $I$-labelled tree $T$, we define $\mathbb{B}\mathbf{P}(T)$ to be the following coend calculated over $U \in \mathsf{Tree}_I$
\[ \mathbb{B}\mathbf{P}(T) := \bar{w}(T;U) \smsh_{U \in \mathsf{Tree}_I} \mathbf{P}(U). \]
The grafting isomorphisms for $\bar{w}(T;U)$, together with the isomorphisms $\mathbf{P}(U) \smsh \mathbf{P}(U') \isom \mathbf{P}(U \cup_i U')$, induce structure maps
\[ \mathbb{B}\mathbf{P}(T) \smsh \mathbb{B}\mathbf{P}(T') \to \mathbb{B}\mathbf{P}(T \cup_i T'). \]
\end{definition}

\begin{proposition}
The structure maps above make $\mathbb{B}\mathbf{P}$ into a pre-cooperad (of spectra or pointed spaces, respectively).
\end{proposition}

\begin{remark} \label{rem:BP}
Suppose $P$ is an operad of unpointed spaces. We write $\mathbb{B}P$ for the pre-cooperad of pointed spaces given by applying $\mathbb{B}$ to the operad $P_+$ of pointed spaces. In this case, we can directly describe a point in $\mathbb{B}P(T)$ as comprising the following data:
\begin{itemize}
  \item an $I$-labelled tree $U$;
  \item a point $r \in [0,\infty]^{E(T)}$; that is a value $r_e \in [0,\infty]$ for each edge $e \in E(T)$;
  \item a point $y \in P(U)$, that is a point $y_{e'} \in P(I_{e'})$ for each edge $e' \in E(U)$.
\end{itemize}
subject to the following identifications:
\begin{enumerate} \itemsep=5pt
  \item The points $(U,r,y)$ and $(U',r,y')$ are identified when
    \begin{itemize}
    \item $U \leq U'$ and $y' \mapsto y$ under the quasi-operad composition map $P(U') \to P(U)$.
    \end{itemize}
  \item The point $(U,r,y)$ is identified with the basepoint in $\mathbb{B}P(T)$ when
\begin{itemize}
  \item $U \nsubseteq T$; or
  \item $r_e = \infty$ for any edge $e \in E(T)$; or
  \item $r_e = 0$ for any edge $e \in E(U)$.
\end{itemize}
\end{enumerate}
\end{remark}

\begin{definition} \label{def:cobar}
Let $\mathbf{Q}$ be a pre-cooperad (of spectra or pointed spaces). Then we define the \emph{cobar construction} $\mathbb{C}\mathbf{Q}$ to be the operad given by the ends
\[ \mathbb{C}\mathbf{Q}(I) := \Map_{U \in \mathsf{Tree}_I}(\bar{w}(U),\mathbf{Q}(U)) \]
with composition maps induced by
\[ \begin{split} \Map(\bar{w}(U),\mathbf{Q}(U)) \smsh \Map(\bar{w}(U'),\mathbf{Q}(U')) &\to \Map(\bar{w}(U) \smsh \bar{w}(U'),\mathbf{Q}(U) \smsh \mathbf{Q}(U')) \\
    &\to \Map(\bar{w}(U \cup_i U'),\mathbf{Q}(U \cup_i U')) \end{split} \]
given by combining the degrafting maps for $\bar{w}$ of Definition~\ref{def:w} with the pre-cooperad structure maps for $\mathbf{Q}$.
\end{definition}

\begin{proposition}
The functors $\mathbb{B}$ and $\mathbb{C}$ form an adjunction between the categories of operads and pre-cooperads, either of spectra or pointed spaces. For a cooperad $\mathbf{Q}$ viewed as a pre-cooperad as in Example~\ref{ex:q-coop}, there is an isomorphism of operads $\mathbb{C}\mathbf{Q} \isom C\mathbf{Q}$ between the cobar construction of Definition~\ref{def:cobar} and that of Proposition~\ref{prop:bar-cobar}.
\end{proposition}

\begin{remark} \label{rem:CB}
Let $P$ be an operad of pointed spaces or spectra. One of the central constructions of \cite{ching:2012} is a map of operads of the form
\[ \theta: WP \to CBP \isom \mathbb{C}BP \]
where $WP$ is the Boardman-Vogt W-construction, and $C$ and $B$ are the cobar and bar constructions introduced previously. Using the adjunction $(\mathbb{B},\mathbb{C})$ we obtain a map of pre-cooperads
\[ \theta^{\#}: \mathbb{B}WP \weq BP. \]
which can be shown to be an equivalence of quasi-cooperads. In fact, for a corolla $\tau_I$, the map
\[ \theta^{\#}_I: \mathbb{B}WP(I) \to BP(I) \]
is an isomorphism. (The above claims follow from the work of \cite{ching:2012} for an operad of spectra, but the constructions in that paper can largely be carried out in the same way in the context of pointed spaces, even though $\mathbb{B}$ and $C$ are no longer a Quillen equivalence in that setting.)
\end{remark}

\section{Proof of the duality theorem} \label{sec:proof}

We are now in position to prove Theorem~\ref{thm:main}. Recall that we constructed in Proposition~\ref{prop:alphah} a map of operads of spectra
\[ \alpha^{\#}: \Sigma^\infty F_V \to \Map(BD_V, \Sigma^\infty \bar{S}_V). \]
First we observe that the right-hand side can be written in a different way using quasi-cooperads.

\begin{lemma} \label{lem:conv}
There is an isomorphism of operads
\[ \Map(BD_V, \Sigma^\infty \bar{S}_V) \isom \mathbb{C}\Map(D_V, \Sigma^\infty \bar{S}_V) \]
where $\Map(D_V, \Sigma^\infty \bar{S}_V)$ is a quasi-cooperad constructed as in Example~\ref{ex:map-qcoop} from the operad $D_V$ and the quasi-cooperad $\Sigma^\infty \bar{S}_V$ associated to the operad structure on $\bar{S}_V$ as in Example~\ref{ex:op-qcoop}.
\end{lemma}
\begin{proof}
The isomorphism is given for a finite set $I$ by the following standard relationship between mapping spectra and ends/coends.
\[ \Map(\bar{w}(T) \smsh_{T \in \mathsf{Tree}_I} D_V(T), \Sigma^\infty \bar{S}_V(I)) \isom \Map_{T \in \mathsf{Tree}_I}(\bar{w}(T),\Map(D_V(T),\Sigma^\infty \bar{S}_V(T))). \qedhere \]
\end{proof}

Combining Lemma~\ref{lem:conv} with $\alpha^{\#}$ we obtain a map of operads
\[ \Sigma^\infty F_V \to \mathbb{C}\Map(D_V,\Sigma^\infty \bar{S}_V) \]
and hence, via the Quillen equivalence $(\mathbb{B},\mathbb{C})$, a map of pre-cooperads
\begin{equation} \label{eq:alpha'} \alpha_{\#}: \Sigma^\infty \mathbb{B}F_V \isom \mathbb{B}(\Sigma^\infty F_V) \to \Map(D_V,\Sigma^\infty \bar{S}_V). \end{equation}
Theorem~\ref{thm:main} now follows from the following claim.

\begin{theorem} \label{thm:qcoop}
The map $\alpha_{\#}$ of (\ref{eq:alpha'}) is an equivalence of quasi-cooperads.
\end{theorem}

We will show two things which together imply this claim: (1) that $\mathbb{B}F_V$ is a quasi-cooperad; (2) that $\alpha_{\#}$ is an equivalence on corollas. The first follows from the comments in Remark~\ref{rem:CB} together with the result of \cite{salvatore:2021} that $F_V \isom WF_V$, but we will give a direct proof. To do this, we first observe that Remark~\ref{rem:BP} permits a fairly explicit description of the pointed spaces $\mathbb{B}F_V(T)$.

\begin{lemma} \label{lem:BF}
Let $T$ be an $I$-labelled tree. The pointed space $\mathbb{B}F_V(T)$ is a quotient of the space
\[ [0,\infty]^{E(T)} \times F_V(I) \]
by the subspace consisting of those points $(r,y)$ for which
\begin{itemize}
  \item $r_e = \infty$ for any $e \in E(T)$; or
  \item $y \in \mathrm{Im}(F_V(I/K) \times F_V(K) \to F_V(I))$ for some $K \subseteq I$ with $|K| \geq 2$, and either $r_K = 0$ or $K \notin E(T)$;
\end{itemize}
\end{lemma}

\begin{lemma}
The pre-cooperad $\mathbb{B}F_V$ is a quasi-cooperad.
\end{lemma}
\begin{proof}
Recall that the operad composition map
\[ \circ_i: F_V(I) \times F_V(J) \to F_V(I \cup_i J) \]
is the inclusion of a face in a manifold with corners. This inclusion extends to a collar neighbourhood of the face in a way that preserves the face structure. In other words there is an inclusion
\[ \circ_i^{\bullet}: [0,\infty] \times F_V(I) \times F_V(J) \to F_V(I \cup_i J) \]
such that: $\circ_i^0 = \circ_i$, and if $y \in F_V(I)$ or $y' \in F_V(J)$ is in the image of some composition map, then $y \circ_i^s y'$ is in the image of the corresponding map, compare \cite{salvatore:2019a}.

Now let $T'$ be an $I$-labelled tree, and $T''$ a $J$-labelled tree. We now use the explicit description in Lemma~\ref{lem:BF} to define a map
\[ \delta: \mathbb{B}F_V(T' \cup_i T'') \to \mathbb{B}F_V(T') \smsh \mathbb{B}F_V(T'') \]
by
\[ \delta(r,y) := \begin{cases} ((r|_{T'},y'), (r|_{T''} + s|_J,y'')) & \text{if $y = y' \circ_i^s y''$}; \\ * & \text{otherwise}. \end{cases} \]
Here $r|_{T''} + s|_J \in [0,\infty]^{E(T'')}$ is given by the restriction of $r \in [0,\infty]^{E(T' \cup_i T'')}$ to the edges of $T''$, with the value $s \in [0,\infty]$ added to the entry corresponding to the root edge $J \in T''$.

We claim that $\delta$ is a pointed homotopy inverse to the pre-cooperad structure map for $\mathbb{B}F_V$. It is straightforward to check that $\delta\gamma$ is the identity on $\mathbb{B}F_V(T') \smsh \mathbb{B}F_V(T'')$. A pointed homotopy from the identity to $\gamma\delta$ is provided by the map
\[ [0,\infty]_+ \smsh \mathbb{B}F_V(T' \cup_i T'') \to \mathbb{B}F_V(T' \cup_i T'') \]
given by
\[ (u,(r,y)) \mapsto \begin{cases} (r + \min\{u,s\}|_J,y' \circ^{se^{-u}}_i y'')) & \text{if $y = y' \circ_i^s y''$}; \\ (r+u|_J,y) & \text{otherwise}. \end{cases} \qedhere \]
\end{proof}

To show that a map $\alpha_{\#}: \Sigma^\infty \mathbb{B}F_V \to \Map(D_V,\Sigma^\infty \bar{S}_V)$ of quasi-cooperads is an equivalence, it is sufficient to show that it is an equivalence on corollas. It follows from Lemma~\ref{lem:BF} that $\mathbb{B}F_V(I) := \mathbb{B}F_V(\tau_I)$ has a simple description: there is a homeomorphism
\[ \mathbb{B}F_V(I) \isom \Sigma F_V(I)/\partial F_V(I) \isom \Sigma \mathring{F}_V(I)^+ \]
where $\partial F_V(I)$ denotes the subspace of decomposable elements of $F_V(I)$, that is, the boundary of this manifold with corners, and $\mathring{F}_V(I)^+$ is the one-point compactification of the open stratum in $F_V(I)$.

Recall that $\mathring{F}_V(I)$ is the space of $I$-indexed configurations in $V$, modulo translation and positive scaling. The suspension coordinate can be used to build the scaling back in, so we can identify $\mathbb{B}F_V(I)$ with the one-point compactification of the configuration space, modulo translation. To make this description more precise, we build a version of the configuration space based on the barycentric operad that also underlies the operads $D_V$ and $\bar{S}_V$.

\begin{definition} \label{def:U}
For a finite set $I$ of cardinality at least $2$, we set
\[ U_V(I) := \{ (x,t) \in R_V(I) \; | \; x_i \neq x_j \text{ for $i \neq j$ in $I$}\}. \]
For each $t \in \Delta(I)$, the fibre $U_V(I)_t$ is therefore the subset of the configuration space of $I$-tuples in $V$ consisting of those configuration that also satisfy the weighted barycentre condition with respect to $t$.

We consider the fibrewise (over $\Delta(I)$) one-point compactification of $U_V(I)$ modulo the section at infinity:
\[ U_V^+(I) := S_V(I)/(S_V(I) - U_V(I)). \]
\end{definition}

\begin{remark}
The spaces $U_V(I)$ do not form a suboperad of $R_V(I)$. However, they are part of what we might call a `quasi-operad' by analogy with our notion of quasi-cooperad. We can define a pre-operad $P$ in a manner dual to that of a pre-cooperad: for each $I$-labelled tree $T$ we have a space $P(T)$; for inclusions $T \subseteq T'$ we have a map $P(T') \to P(T)$, and we have `degrafting maps' $P(T \cup_i T') \to P(T) \times P(T')$. A `quasi-operad' is a pre-operad for which the degrafting maps are weak equivalences. This definition is in fact a special case of the dendroidal Segal spaces of Cisinski and Moerdijk~\cite{cisinski/moerdijk:2013} (restricted to trees with no unary or nullary vertices).

The operad $R_V$ is, in particular, a quasi-operad, and it admits a sub-quasi-operad $U_V$ given on corollas by the spaces of Definition~\ref{def:U}. In a similar manner the fibrewise one-point compactifications $U_V^+(I)$ form part of a quasi-cooperad of pointed spaces given by
\[ U_V^+(T) := S_V(T)/(S_V(T) - U_V(T)). \]
\end{remark}

\begin{lemma} \label{lem:phi}
There is a homeomorphism
\[ \phi: \Delta(I)_+ \smsh \mathbb{B}F_V(I) \isom U_V^+(I) \]
given by
\[ (t,(r,y)) \mapsto (ry(I)_i)_{i \in I} \]
where the component $I$-tuple $y(I)$ of the point $y \in F_V(I)$ is chosen to satisfy the weighted barycentre and norm conditions of Remark~\ref{rem:FM} with respect to the point $t \in \Delta(I)$.
\end{lemma}
\begin{proof}
An inverse to $\phi$ is given by the following construction. Given $(x,t) \in U_V(I)$, there is a unique $r \in (0,\infty)$ such that $\frac{1}{r}x$ satisfies the weighted norm condition with respect to $t$. We therefore map $(x,t)$ to the point $(t,(r,[x]))$ in $\Delta(I) \times \mathbb{B}F_V(I)$, where $[x]$ denotes the point in $\mathring{F}_V(I) \subseteq F_V(I)$ determined by the configuration $x$. This construction extends continuously to the basepoint in $U_V^+(I)$.
\end{proof}

It remains to show that the fibrewise one-point compactification $U_V^+(I)$ is a model for the Spanier-Whitehead dual of the restricted little disc space $D_V(I)$. We see this as the (fibrewise) application of the following general Spanier-Whitehead duality result of Dold and Puppe:

\begin{theorem}[Dold-Puppe~\cite{dold/puppe:1980}] \label{thm:dold/puppe}
Let $U$ be the complement of a finite cell complex in $S^n$, $D \subseteq U$ a finite cell complex, and $B$ a star-shaped open neighbourhood of the origin in $\R^n$, such that:
\begin{itemize}
  \item the inclusion $D \subseteq U$ is a weak homotopy equivalence;
  \item vector addition in $\R^n$ restricts to a map $D \times B \to U$.
\end{itemize}
Let $U^+ = S^n/(S^n - U)$ and $B^+ = S^n/(S^n - B)$ be the one-point compactifications of $U$ and $B$ respectively. Then the map
\[ \sigma: D_+ \smsh U^+ \to B^+ \homeq S^n \]
given by
\[ (z,x) \mapsto z-x \]
is an $n$-duality evaluation map. In other words, $\sigma$ induces an equivalence of spectra
\[ \Sigma^\infty U^+ \weq \Map(D_+,\Sigma^\infty B^+). \]
In particular, since $B^+ \homeq S^n$, it follows that
\[ \Sigma^\infty U^+ \homeq \Sigma^n\dual(D_+). \]
\end{theorem}
\begin{proof}
We have the following commutative diagram
\[ \begin{diagram}
  \node{D_+ \smsh (\R^n \cup C(\R^n - D))} \arrow{e} \arrow{s,l}{\sim} \node{\R^n \cup C(\R^n - \{0\})} \arrow{s,r}{\sim} \\
  \node{D_+ \smsh (S^n \cup C(S^n - D))} \arrow{e} \node{S^n \cup C(S^n - \{0\})} \\
  \node{D_+ \smsh (S^n \cup C(S^n - U))} \arrow{e} \arrow{n,l}{\sim} \arrow{s,l}{\sim} \node{S^n \cup C(S^n - B)} \arrow{s,r}{\sim} \arrow{n,r}{\sim}  \\
  \node{D_+ \smsh U^+} \arrow{e} \node{B^+}
\end{diagram} \]
where all the horizontal maps are given by $(z,x) \mapsto z-x$, and $A \cup CB$ denotes the mapping cone of the inclusion $B \subseteq A$ (with cone point as the basepoint). The top and middle vertical maps are induced by inclusions of subsets, and the bottom vertical maps are given by collapsing cones to the basepoint.

The top horizontal map is an $n$-duality evaluation map by \cite[3.6]{dold/puppe:1980}, so it is sufficient to show that the vertical maps are stable equivalences. This is clear for the right-hand column of maps. For the top-left vertical map, it follows from the fact that
\[ \R^n - D \to S^n - D \to S^n \]
is a homotopy cofibre sequence (of unpointed spaces) when $D$ is bounded. The bottom-left vertical map is an equivalence by our second assumption, since then the inclusion $S^n - U \subseteq S^n$ is a cofibration and $S^n/(S^n - U) \isom U^+$. Finally, the middle-left vertical map is a stable equivalence because the inclusion $S^n - U \to S^n - D$ is $(n-1)$-dual to the inclusion $D \to U$ by Alexander duality.
\end{proof}

We apply Theorem~\ref{thm:dold/puppe} in a fibrewise manner to the inclusion $D_V(I) \subseteq U_V(I)$ between subspaces of the vector bundle $R_V(I) \to \Delta(I)$. The role of the open subset $B$ is played by the space $\mathring{S}_V(I)$ of Definition~\ref{def:S}.

\begin{definition} \label{def:psi}
We write
\[ D_V(I)_+ \smsh_{\Delta(I)} U_V^+(I) \]
for the subspace of the smash product $D_V(I)_+ \smsh U_V^+(I)$ consisting of those pairs $((z,t),(x,u))$ for which $t = u$, in addition to the basepoint. We then define a map
\[ \psi: D_V(I)_+ \smsh_{\Delta(I)} U_V^+(I) \to \bar{S}_V(I) \]
by
\[ \psi((z,t),(x,t)) := (z-x,t). \]
The map $\psi_T$ is well-defined because the (fibrewise) vector addition in $R_V(I)$ restricts to a map
\[ D_V(I) \times \mathring{S}_V(I) \to U_V(I). \]
To see this claim, suppose $(z,t) \in D_V(I)$ and $(w,t) \in \mathring{S}_V(I)$. Then
\[ |(z+w)_i - (z+w)_j| \geq |z_i - z_j| - |w_i - w_j| > (t_i + t_j) - \min\{t_i,t_j\} > 0. \]
Note also that the inclusion $D_V(I) \subset U_V(I)$ is a (fibrewise) weak equivalence by Theorem~\ref{thm:D}. Thus the conditions of Theorem~\ref{thm:dold/puppe} are satisfied. It follows that $\psi_I$ is a fibrewise $S$-duality map. We are now in position to complete the proof of Theorem~\ref{thm:qcoop} and hence of the main result of this paper.
\end{definition}

\begin{proof}[Proof of Theorem~\ref{thm:qcoop}]
Combining the maps $\phi$ of Lemma~\ref{lem:phi} and $\psi$ of Definition~\ref{def:psi}, we obtain a map
\[ \rho: \mathbb{B}F_V(I) \smsh D_V(I)_+ \to \bar{S}_V(I) \]
given by
\[ ((r,y), (z,t)) \mapsto (z - ry(I),t). \]
This in turn induces a map of spectra
\[ \rho_{\#}: \Sigma^\infty \mathbb{B}F_V(I) \to \Map(D_V(I),\Sigma^\infty \bar{S}_V(I)) \]
which we claim is precisely the map $\alpha_{\#}$ of (\ref{eq:alpha'}) applied to the corolla $\tau_I$. This claim follows easily from the simple description of the map $\alpha_T$ of Definition~\ref{def:alpha} in the case that $T$ is a corolla, so that for any $i \in I$ the relevant sum consists of a single term corresponding to the root edge $I$, for which $t_I = 1$.

To see that $\rho_{\#}$, and hence $\alpha_{\#}$, is an equivalence of spectra, fix some $t \in \Delta(I)$ and consider the following diagram (where the subscripts `$t$' denote the relevant fibre over $t \in \Delta(I)$):
\[ \begin{diagram}
  \node{\Sigma^\infty \mathbb{B}F_V(I)} \arrow{e,tb}{\phi(t,-)}{\sim} \arrow[2]{s,l}{\rho_{\#}} \node{\Sigma^\infty U_V^+(T)_t} \arrow{s,lr}{\sim}{\psi_t} \\
  \node[2]{\Map(D_V(T)_t,\Sigma^\infty \bar{S}_V(I)_t)} \arrow{s,r}{\sim} \\
  \node{\Map(D_V(T),\Sigma^\infty \bar{S}_V(I))} \arrow{e,t}{\sim} \node{\Map(D_V(I)_t,\Sigma^\infty \bar{S}_V(I))}
\end{diagram} \]
where the bottom horizontal map is restriction along the inclusion $D_V(T)_t \to D_V(T)$, and the bottom-right vertical map is induced by the inclusion $\bar{S}_V(I)_t \to \bar{S}_V(I)$.

The map $\phi(t,-)$ is a homeomorphism by Lemma~\ref{lem:phi} and $\psi_t$ is an equivalence of spectra by Theorem~\ref{thm:dold/puppe}. It remains to show that each of the two inclusions mentioned above is a weak equivalence.

For $D_V(T)_t \to D_V(T)$ this follows from Proposition~\ref{prop:D}. For $\bar{S}_V(I)_t \to \bar{S}_V(I)$, we note the following commutative diagram of pointed spaces
\[ \begin{diagram}
  \node{S_V(I)_t} \arrow{s,l}{\sim} \arrow{e,t}{\sim} \node{S_V(I)} \arrow{s,r}{\sim} \\
  \node{\bar{S}_V(I)_t} \arrow{e} \node{\bar{S}_V(I)}
\end{diagram} \]
where the top horizontal map is an equivalence because $S_V(I)$ is the Thom space of the trivial vector bundle $R_V(I)$ over $\Delta(I)$, so that map is an inclusion of the form
\[ S^n \to S^n \smsh \Delta(I)_+. \qedhere \]
\end{proof}

\section{Compatibility with embeddings of vector spaces} \label{sec:compatibility}

We now turn to the compatibility of the equivalences of our main results with those maps induced by a linear embedding of one vector space in another. To be more precise, we look at the embedding of a normed vector space $V$ into a direct sum $V \oplus W$ where the norm on the direct sum is given as follows.

\begin{definition}
Given normed vector spaces $V,W$, we use the norm on $V \oplus W$ given by
\[ |(v,w)| := \max\{|v|,|w|\}. \]
\end{definition}

\begin{definition} \label{def:D-susp}
Let $V,W$ be finite-dimensional normed vector spaces, and let $I$ be a finite set. Then we define a map of pointed spaces
\[ \kappa_I: D_{V \oplus W}(I)_+ \to D_V(I)_+ \smsh_{\Delta(I)} \bar{S}_W(I); \quad ((x,y),t) \mapsto ((x,t), (y,t)). \]
To see this produces a well-defined continuous map into the fibrewise smash product, suppose that $((x,y),t) \in D_{V \oplus W}(I)$ and $(y,t) \in \mathring{S}_W(I)$. We want to show that $(x,t) \in D_V(I)$. For each $i \in I$, we have
\[ \max\{|x_i|,|y_i|\} = |(x_i,y_i)| \leq 1-t_i \]
and so $|x_i| \leq 1-t_i$. For each pair $i,j \in I$, we have
\[ \max\{|x_i-x_j|,|y_i-y_j|\} = |(x_i,y_i) - (x_j,y_j)| \geq t_i + t_j. \]
But $|y_i - y_j| < \min\{t_i,t_j\} < t_i+t_j$ and so we must have $|x_i - x_j| \geq t_i+t_j$. Therefore $(x,t) \in D_V(I)$ as desired.
\end{definition}

\begin{proposition} \label{prop:kappa}
 The maps of Definition~\ref{def:D-susp} form an $O(V) \times O(W)$-equivariant map of operads (and hence quasi-operads) of pointed spaces
\[ \kappa: {D_{V \oplus W}}_+ \to {D_V}_+ \smsh_{\Delta} \bar{S}_W. \]
\end{proposition}
\begin{proof}
This claim follows from linearity of the operad structure maps for the overlapping discs operad.
\end{proof}

\begin{remark} \label{rem:VW}
There is an equivalence of operads $S_W \weq \bar{S}_W$ and so the target of the map $\kappa$ is equivalent to the operadic suspension $\Sigma^W{D_V}_+$. We thus think of $\kappa$ as a model for a suitable map of operads (of pointed spaces)
\[ {E_{V \oplus W}}_+ \to  \Sigma^W{E_V}_+. \]
Maps of this type are originally due to Peter May~\cite{may:1972}, and have also been studied by Ahearn and Kuhn~\cite[\S7]{ahearn/kuhn:2002}.
\end{remark}

\begin{proposition} \label{prop:beta}
Let $V,W$ be finite-dimensional normed vector spaces. Then there is an isomorphism of operads of pointed spaces
\[ S_V \smsh_{\Delta} S_W \isom S_{V \oplus W} \]
given by $((x,t),(y,t)) \mapsto ((x,y),t)$ and which induces an equivalence of operads (and hence of quasi-cooperads):
\[ \sigma: \bar{S}_V \smsh_{\Delta} \bar{S}_W \weq \bar{S}_{V \oplus W}. \]
\end{proposition}
\begin{proof}
The isomorphism of vector bundles $R_V(I) \oplus R_W(I) \isom R_{V \oplus W}(I)$, over $\Delta(I)$, induces the desired homeomorphisms between Thom spaces. To see that these maps pass to the quotient, we have to check that if $((x,y),t) \notin \mathring{S}_{V \oplus W}(I)$, then either $(x,t) \notin \mathring{S}_V(I)$ or $(y,t) \notin \mathring{S}_W(I)$.

So suppose that $(x,t) \in \mathring{S}_V(I)$ and $(y,t) \in \mathring{S}_W(I)$. For each $i \in I$, we have
\[ |(x_i,y_i)| = \max\{|x_i|,|y_i|\} < t_i \]
and for $i,j \in I$, we have
\[ |(x_i,y_i)-(x_j,y_j)| = |(x_i-x_j,y_i-y_j)| = \max\{|x_i-x_j|,|y_i-y_j|\} < \min\{t_i,t_j\} \]
so $((x,y),t) \in \mathring{S}_{V \oplus W}(I)$ as required.
\end{proof}

\begin{theorem} \label{thm:pro}
Let $V$ and $W$ be finite-dimensional normed vector spaces. Then there is an $O(V) \times O(W)$-equivariant commutative diagram of quasi-cooperads:
\[ \begin{diagram}
  \node{\Sigma^\infty\mathbb{B}F_V} \arrow[2]{s} \arrow{e,tb}{\alpha_{\#}(V)}{\sim} \node{\Map({D_V}_+,\Sigma^\infty \bar{S}_V)} \arrow{s,r}{\sim} \\
  \node[2]{\Map({D_V}_+ \smsh_{\Delta} \bar{S}_W, \Sigma^\infty \bar{S}_V \smsh_{\Delta} \bar{S}_W)} \arrow{s,r}{\Map(\kappa,\sigma)} \\
  \node{\Sigma^\infty\mathbb{B}F_{V \oplus W}} \arrow{e,tb}{\alpha_{\#}(V \oplus W)}{\sim} \node{\Map({D_{V \oplus W}}_+,\Sigma^\infty \bar{S}_{V \oplus W})}
\end{diagram} \]
where
\begin{itemize}
  \item the left-hand vertical map is induced by the inclusion $F_V \to F_{V \oplus W}$;
  \item the top-right vertical map is the map of spectra induced by maps of pointed spaces
  \[ \Map(D_V(T)_+,S^n \smsh \bar{S}_V(T)) \to \Map(D_V(T)_+ \smsh_{\Delta} \bar{S}_W(T), S^n \smsh \bar{S}_V(T) \smsh \bar{S}_W(T)); \]
  because each space $\bar{S}_W(T)$ is homotopy equivalent to a sphere, this map of spectra is a stable equivalence;
  \item the bottom-right vertical map is induced by the map of operads $\kappa$ and the map of quasi-cooperads $\sigma$.
\end{itemize}
\end{theorem}
\begin{proof}
This claim is a simple diagram chase.
\end{proof}

\begin{corollary}
Let $V$ and $W$ be finite-dimensional normed vector spaces. Then the Koszul dual of the inclusion of stable (reduced) little-disc operads
\[ \mathbf{E}_V \to \mathbf{E}_{V \oplus W} \]
can be identified, under the equivalences of Theorem~\ref{thm:main}, with the $V \oplus W$-desuspension of the operad map
\[ \mathbf{E}_{V \oplus W} \to \Sigma^{W}\mathbf{E}_V \]
described in Remark~\ref{rem:VW}.
\end{corollary}

\begin{remark}
We now consider the normed vector spaces $\R^n$ with the $\ell^\infty$-norm, and when $V = \R^n$ we abbreviate $D_V$, $B_V$, $S_V$, etc... as $D_n$, $B_n$, $S_n$, etc... Thus $E_n$ denotes the ordinary little $n$-\emph{cubes} operad, and $\mathbf{E}_n$ the corresponding operad of spectra.

The sequence of inclusions of the form $x \mapsto (x,0)$
\[ \R^1 \to \R^2 \to \R^3 \to \dots \]
then determines a sequence of operads of spectra
\[ \mathbf{E}_1 \to \mathbf{E}_2 \to \mathbf{E}_3 \to \dots \]
which, on applying Koszul duals, gives us an inverse sequence of operads
\[ K\mathbf{E}_1 \leftarrow K\mathbf{E}_2 \leftarrow K\mathbf{E}_3 \leftarrow \dots \]
We would now like to identify this sequence with something of the form
\[ \Sigma^{-\R^1}\mathbf{E}_1 \leftarrow \Sigma^{-\R^2}\mathbf{E}_2 \leftarrow \Sigma^{-\R^3}\mathbf{E}_3 \leftarrow \dots \]
but it is a little tricky to make the maps in this sequence precise. Here is one approach.
\end{remark}

\begin{definition} \label{def:kappa*}
In the formulas below, we use $\smsh$ to denote the fibrewise smash product $\smsh_{\Delta}$ of Definition~\ref{def:psi}. We also write $\mathbf{D}_n := \Sigma^\infty_+D_n$ for the stable (restricted, reduced) little $n$-cubes operad.

For each $n \in \mathbb{N}$, we have operads of spectra
\[ \tilde{\Sigma}^{-\R^n}\mathbf{D}_n := \hocolim_r \Sigma^{-\R^{n+r}}(\mathbf{D}_n \smsh (\bar{S}_1)^{\smsh r}) \]
where the maps in this homotopy colimit are all stable equivalences of operads and take the form
\[ \begin{split} \Map(S_{n+r}, \mathbf{D}_n \smsh (\bar{S}_1)^{\smsh r})
    &\weq \Map(S_{n+r} \smsh S_1, \mathbf{D}_n \smsh (\bar{S}_1)^{\smsh r} \smsh S_1) \\
    &\weq \Map(S_{n+r+1}, \mathbf{D}_n \smsh (\bar{S}_1)^{\smsh r+1}) \end{split} \]
with the first map of a similar nature to the top-right vertical map in the diagram in Theorem~\ref{thm:pro}, and the second induced by the isomorphism of cooperads $S_{n+r+1} \isom S_{n+r} \smsh S_1$, the equivalence of operads $S_1 \weq \bar{S}_1$, and Proposition~\ref{prop:beta}. The canonical map to the homotopy colimit provides an equivalence of operads
\[ \Sigma^{-\R^n}\mathbf{D}_n \weq \tilde{\Sigma}^{-\R^n}\mathbf{D}_n. \]

We then also have maps of operads
\[ \kappa_n^*: \tilde{\Sigma}^{-\R^{n+1}}\mathbf{D}_{n+1} \to \tilde{\Sigma}^{-\R^n}\mathbf{D}_n \]
given on the $r$th term in the homotopy colimit by
\[ \Sigma^{-\R^{n+1+r}}(\mathbf{D}_{n+1} \smsh (\bar{S}_1)^{\smsh r}) \to \Sigma^{-\R^{n+1+r}}(\mathbf{D}_n \smsh (\bar{S}_1)^{\smsh 1+r}) \]
induced by the map $\kappa: \mathbf{D}_{n+1} \to \mathbf{D}_n \smsh \bar{S}_1$ of Proposition~\ref{prop:kappa}.
\end{definition}

\begin{theorem} \label{thm:seq}
The inverse sequence of operads
\[ K\mathbf{E}_1 \leftarrow K\mathbf{E}_2 \leftarrow K\mathbf{E}_3 \leftarrow \dots \]
is equivalent to the sequence
\[ \tilde{\Sigma}^{-\R^1}\mathbf{D}_1 \leftarrow \tilde{\Sigma}^{-\R^2}\mathbf{D}_2 \leftarrow \tilde{\Sigma}^{-\R^3}\mathbf{D}_3 \leftarrow \dots \]
consisting of the operad maps $\kappa_n^*$ of Definition~\ref{def:kappa*}.
\end{theorem}

The homotopy limit of the inverse sequence $(K\mathbf{E}_n)_{n \in \mathbb{N}}$ is equivalent to $K\mathbf{E}_\infty \homeq K\mathbf{Com}$, i.e.\ the spectral Lie-operad. Theorem~\ref{thm:seq} thus provides us with a new model for the spectral Lie-operad, hence also the Goodwillie derivatives of the identity, that does not involve the bar construction.

\begin{corollary} \label{cor:lie}
There is an equivalence of operads of spectra
\[ \mathbf{Lie} \homeq \holim_n \tilde{\Sigma}^{-\R^n}\mathbf{D}_n. \]
\end{corollary}


\providecommand{\bysame}{\leavevmode\hbox to3em{\hrulefill}\thinspace}
\providecommand{\MR}{\relax\ifhmode\unskip\space\fi MR }
\providecommand{\MRhref}[2]{%
	\href{http://www.ams.org/mathscinet-getitem?mr=#1}{#2}
}
\providecommand{\href}[2]{#2}

\end{document}